\documentclass{amsart}

\usepackage{bbm}	
\usepackage{mathrsfs}	

\usepackage{amsmath, amsthm, amssymb}	
\usepackage{enumitem}	

\usepackage{tikz}	

\setlist[enumerate]{label=(\roman*)}

\newtheorem{theorem}{Theorem}[section]
\newtheorem{lemma}[theorem]{Lemma}

\theoremstyle{definition}
\newtheorem{definition}[theorem]{Definition}
\newtheorem{example}[theorem]{Example}

\theoremstyle{remark}

\numberwithin{equation}{section}

\newcommand{\textdef}{\textit}

\newcommand{\N}{\mathbb{N}}
\newcommand{\Q}{\mathbb{Q}}
\newcommand{\R}{\mathbb{R}}

\newcommand{\1}{\mathbbm{1}}	

\newcommand{\PV}{\mathbb{P}}	
\newcommand{\EV}{\mathbb{E}}	

\newcommand{\cC}{\mathcal{C}}	

\newcommand{\sE}{\mathscr{E}}	
\newcommand{\sF}{\mathscr{F}}	
\newcommand{\sG}{\mathscr{G}}	
\newcommand{\sS}{\mathscr{S}}	

\newcommand{\tE}{\widetilde{E}}
\newcommand{\tsE}{\widetilde{\sE}}

\newcommand{\hO}{\hat{\Omega}}
\newcommand{\hE}{\hat{E}}
\newcommand{\hsE}{\hat{\sE}}
\newcommand{\hsF}{\hat{\sF}}

\newcommand{\hPV}{\hat{\PV}}
\newcommand{\hEV}{\hat{\EV}}

\newcommand{\hT}{\hat{\Theta}}

\newcommand{\norm}[1]{\left\lVert #1 \right\rVert} 

\newcommand\restr[2]{{		  	
  \left.\kern-\nulldelimiterspace 
  #1 				  
  \vphantom{\big|}		  
  \right|_{#2}    		  
  }}

\begin{document}

\title{Concatenation and Pasting of Right Processes}

\author{Florian Werner}
\address{Institut f\"ur Mathematik, Universit\"at Mannheim, 68131 Mannheim, Germany}
\email{fwerner@math.uni-mannheim.de}



\date{\today}



\begin{abstract}
  A universal method for the concatenation of a sequence of Markov right processes is established.
It is then applied to the continued pasting of two Markov right processes, which can be used for pathwise constructions
of locally defined processes like Brownian motions on compact intervals.
\end{abstract}

\maketitle

\section{Introduction}

\subsection{The Objective}

The concatenation of a sequence of (strong) Markov processes $(X^n, n \in \N)$ on state spaces $(E^n, n \in \N)$ forms a stochastic process~$X$ on~$\bigcup_{n \in \N} E^n$ as follows:
Started in~$E^n$, the process~$X$ behaves like~$X^n$ until this process dies, afterwards is revived as~$X^{n+1}$ at a point in~$E^{n+1}$ 
which is chosen by a probability measure which takes Markovian information of~$X^n$ until its death into account, then behaves like~$X^{n+1}$ until it dies, and so on.

In earlier works on Markov processes and their applications, 
 the theory of this technique, in contrast to other well-known modes of transformation like killing or time substitution,
 has not been developed much further---if at all---than on restricting it to special cases,
despite the fact that it is not at all trivial to show that the resulting process $X$ will inherit the (strong) Markov property of the subprocesses.
This gap in the literature is quite surprising, considering it is natural in manifold applications to construct processes via local solutions and pasting them together,
from immediate constructions of Markov chains and branching processes \cite{IkedaNagasawaWatanabe68},
extending Markov processes over their lifetime by instant revivals \cite{Meyer75},
introduction of isolated jump discontinuities into diffusion processes, 
up to the pathwise construction of stochastic processes via local solution techniques such as in the construction
of Brownian motions on intervals~\cite{ItoMcKean63,Knight81} or on metric graphs \cite{KPS12,FitzsimmonsKuterGraph,Werner16}.

In this paper, we are establishing the technique of concatenation of countably many processes in the general context of right processes \cite{Sharpe88}. 
This class of strong Markov processes encompasses a majority of classical types of Markov processes,
such as Feller, Hunt, standard, and---in some sense \cite{Getoor75}---even Ray processes.
Our main result will guarantee that the process constructed by the concatenation of a sequence of right processes on disjoint state spaces 
via transfer kernels will again be a right process, thus especially maintaining the strong Markov property of its subprocesses. 
This generalizes~\cite{Sharpe88} from two to countably many processes.
We will then weaken the assumption on the disjointedness of the state spaces to the concatenation of alternating copies of two right processes 
by imposing some consistency conditions on both partial processes. 
This method can be used to glue two Markov processes on not necessarily disjoint state spaces together, extending a result of \cite{Nagasawa76},
or to form instant revival processes in the sense of \cite{IkedaNagasawaWatanabe66,Meyer75}.
We thus provide an unified way to extend or join an extensive class of Markov processes.

\subsection{The Context: Markov Right Processes \& Strong Markov Property}

We understand a Markov process $X$ on a Radon space $E$ (equipped with a $\sigma$-algebra~$\sE$)  to be defined   
in the canonical sense of the standard works of Dynkin~\cite{Dynkin65}, Blumenthal--Getoor~\cite{BlumenthalGetoor69} and Sharpe~\cite{Sharpe88},
that is, as a sextuple 
 \begin{align*}
  X = \big( \Omega, \sG, (\sG_t, t \geq 0), (X_t, t \geq 0), (\Theta_t, t \geq 0), (\PV_x, x \in E) \big)
 \end{align*}
with the following properties:
$(X_t, t \geq 0)$ is a right continuous, $E$-valued stochastic process on the measurable space $(\Omega, \sG)$,
adapted to the filtration $(\sG_t, t \geq 0)$, and equipped with shift operators $(\Theta_t, t \geq 0)$ on $\Omega$.
$(\PV_x, x \in E)$ is a family of probability measures satisfying $X_0 = x$ $\PV_x$-a.s.\ for all $x \in E$ (normality of the process),
such that for all $t \geq 0$, $B \in \sE$, $x \mapsto \PV_x(X_t \in B)$ is measurable
and the Markov property holds:\footnote{For any $\sigma$-algebra $\sE$, we define $b\sE$, $p\sE$ to be
  the sets of all $\sE$-measurable functions which are bounded, non-negative respectively, 
  as well as $bp\sE := b\sE \cap p\sE$.}\textsuperscript{,}\footnote{For convenience,
  we omit the qualifier ``a.s.'' in equations containing conditional expectations.}
 \begin{align*}
  \forall x \in E, s, t \geq 0, f \in b\sE: \quad \EV_x \big( f(X_{s+t}) \,\big|\, \sG_s \big) = \EV_{X_s} \big( f(X_t) \big).
 \end{align*}
 
We are basing our results in the context of one of the most general classes of Markov processes, namely the class of right processes.
Right processes are Markov processes which satisfy the following condition of right continuity in the topology of excessive functions:
For $\alpha \geq 0$, the class $\sS_\alpha$ of $\alpha$-excessive functions is the set of all non-negative, measurable functions
which satisfy $e^{-\alpha t} \, T_t f \uparrow f$ pointwise as~$t \downarrow 0$, with $(T_t, t \geq 0)$ being the semigroup associated to $X$, 
that is
 \begin{align*}
  T_t f(x) := \EV_x \big( f(X_t) \big), \quad f \in p\sE \cup b\sE, x \in E.
 \end{align*}
Then a Markov process $X$, equipped with an augmented and right continuous filtration, is called right process,
if it satisfies 
 \begin{align} \label{eq:HD2}
   \text{for all $\alpha > 0$, $f \in \sS_\alpha$, the map $t \mapsto f(X_t)$ is a.s.\ right continuous.} \tag{HD2}
 \end{align}
It is well-known (see \cite[Theorem 7.4]{Sharpe88}) that in order to establish \eqref{eq:HD2}, it is sufficient to check
the right continuity of the process on the $\alpha$-potentials ${(U_\alpha, \alpha > 0)}$
 \begin{align*}
  U_\alpha f(x) := \int_0^\infty e^{-\alpha t} \, T_t f(x) \, dt = \EV_x \Big( \int_0^\infty  e^{-\alpha t} \, f(X_t) \, dt \Big), \quad f \in p\sE \cup b\sE, x \in E,
 \end{align*}
of bounded, uniformly continuous functions\footnote{The set of all bounded and uniformly continuous functions on $E$ is denoted by $b\cC_d(E)$.} on $E$. 
Furthermore, \eqref{eq:HD2} implies the strong Markov property of the process [loc.~cit.], that is, for every $(\sG_t, t \geq 0)$-stopping time $\tau$, 
with $\sF$ being the universal completion of $\sigma(X_s, s \geq 0)$:
 \begin{align*}
    \forall x \in E, Y \in b\sF: \quad \EV_x \big( Y \circ \Theta_\tau \, \1_{\{\tau < \infty\}} \,\big|\, \sG_{\tau+} \big) = \EV_{X_\tau} \big( Y \big) \, \1_{\{\tau < \infty\}}.
 \end{align*}
The strong Markov property is oftentimes crucial for the examination of stochastic processes, 
 in particular it allows to decompose the resolvent of a strong Markov process~$X$ at stopping times $\tau$ via Dynkin's formula \cite[Section 5.1]{Dynkin65}:
 \begin{align}\label{eq:Dynkins formula (resolvent)}
  U_\alpha f(x) = \EV_x \Big( \int_0^\tau  e^{-\alpha t} \, f(X_t) \, dt \Big) + \EV_x \big( e^{-\alpha \tau} \, U_\alpha f(X_\tau) \, \1_{\{\tau < \infty\}} \big).
 \end{align}

We impose the usual hypotheses 
(cf.\ \cite[Sections 3--8,~11,~A1]{Sharpe88}): 
$\sE$ is the universal completion of the Borel $\sigma$-algebra on $E$,
the underlying filtration $(\sG_t, t \geq 0)$ is augmented and right continuous, and there exists an
isolated, absorbing cemetery state $\Delta \in E$, such that with the lifetime of the process
  \begin{align*}
   \zeta := \inf \{ t \geq 0: X_t = \Delta \},
  \end{align*}
$X_{t} = \Delta$ holds for all $t \geq \zeta$. Furthermore, there is a dead path $[\Delta] \in \Omega$ with $\zeta([\Delta]) = 0$, and we constitute that $f(\Delta) = 0$
for any measurable function $f$, which in conjunction with $X_{\infty} := \Delta$, $\Theta_{\infty} := [\Delta]$ allows to drop the restricting functions $\1_{\{\tau < \infty\}}$ in the above formulas
of the strong Markov property.

\subsection{Concatenation of Processes: Construction Approach \& Main Result} \label{sec:intro concatenation}

 \begin{figure}[tb] 
  \centering
  \input{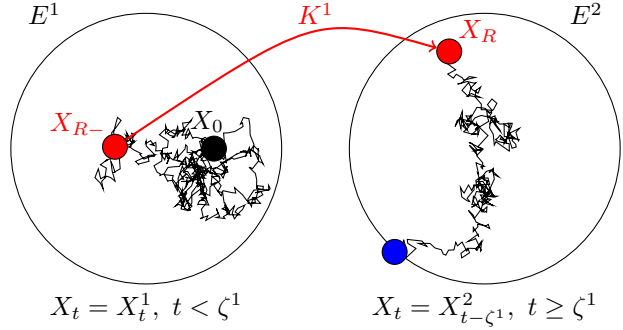}
  \caption[Concatenation of two processes]
           {Concatenation of two processes $X^1$ and $X^2$ on $E^1$, $E^2$, resulting in the process $X$, which, if started in $E^1$, behaves like $X^1$ until $R = \zeta^1$,
            afterwards is revived on some point in $E^2$ (chosen by a transfer kernel $K^1$), where it then runs like $X^2$.} \label{fig:concatenation of 2}
 \end{figure}
 
Let ${(X^n, n \in \N)}$ be a sequence of right processes on disjoint spaces $(E^n, n \in \N)$.
For the pathwise definition of a concatenating process $X$ on $\Omega := \prod_{n \in \N} \Omega^n$, 
we set, for $\omega := (\omega^n, n \in \N) \in \Omega$, $t \geq 0$, 
 \begin{align*}
  X_t(\omega) :=
    \begin{cases}
      X^1_t(\omega^1),				 & t < \zeta^{1}(\omega^1), \\
      X^2_{t - \zeta^{1}(\omega^1)}(\omega^2),		 & \zeta^{1}(\omega^1) \leq t < \zeta^{1}(\omega^1) + \zeta^{2}(\omega^2), \\
      X^3_{t - \left(\zeta^{1}(\omega^1) + \zeta^{2}(\omega^2)\right)}(\omega^3),		 & \zeta^{1}(\omega^1) + \zeta^{2}(\omega^2) \leq t < \zeta^{1}(\omega^1) + \cdots + \zeta^{3}(\omega^3), \\
      ~ \vdots					 & ~ \vdots \\
      \Delta,  & t \geq \sum_{n \in \N}\zeta^{n}(\omega^n).
    \end{cases}
 \end{align*}
 
In order to define initial measures $(\PV_x, x \in E)$ for the process $X$, we need to constitute a transfer mechanism between the
subprocesses $(X^n, n \in \N)$, more precisely: a law on how the process $X^{n+1}$ initiates in $E^{n+1}$ after $X^n$ died.
This mechanism can depend on all information until the lifetime $\zeta^n$ of the subprocess $X^n$,
but it should admit a memoryless property in order to ensure the Markov property of the resulting process $X$.
The main principle which allows to salvage the Markov property is the following invariance under time shifts:

\begin{definition} \label{def:left germ field}
 For a right process $X$ on $E$ and a terminal time $T$ for $X$,
 the \textdef{left germ field} $\sF_{[T-]}$ for $X$ at $T$ consists of all $\sF_{T-}$-measurable random variables $H$ 
 which satisfy
  \begin{align*}
    \forall t \geq 0: \quad   H \circ \Theta_t = H  \text{ a.s.\ on $\{t < T\}$}.
  \end{align*}
\end{definition}

Here, terminal times are a well-known concept for memoryless stopping times:

\begin{definition} \label{def:A_SM:terminal time}
 A stopping time $T$ over $(\sF_t, t \geq 0)$ is a \textdef{terminal time} for a Markov process $X$,
 provided that
  \begin{align*}
    t + T \circ \Theta_t = T \quad \text{on $\{t < T\}$}.
  \end{align*}
\end{definition}

The prime examples for terminal times are the first entrance times. 
Most notably, the lifetime $\zeta$ of a right process is always a terminal time. 
As $\Delta$ is absorbing, we even have a stronger version of shift invariance of $\zeta$ for any random time $R$:
\begin{align} \label{eq: lifetime shift}
 \zeta \circ \Theta_R = (\zeta - R) \vee 0.
\end{align}

The revival information is then encoded in kernels which are memoryless with respect to 
the lifetimes of the partial processes:

\begin{definition} \label{def:transfer kernel}
 Let $X^1$, $X^2$ be right processes on $E^1$, $E^2$ respectively.
 $K$ is a \textdef{transfer kernel} from $X^1$ to $(X^2, E^2)$, 
 if it is a probability kernel from $(\Omega^1, \sF^1_{[\zeta^1-]})$ to $(E^2, \sE^2)$.
\end{definition}

With the help of transfer kernels $K^n$ from $X^n$ to $(X^{n+1}, E^{n+1})$, the paths of the concatenated process are chosen for any $x \in E^n$, $n \in \N$,
by the initial measure
 \begin{align*}
   & \PV_x(d\omega^1, \ldots, d\omega^{n-1}, d\omega^n, d\omega^{n+1}, \ldots) \\
   & := \delta_{[\Delta^1]}(d\omega^1) \cdots \delta_{[\Delta^{n-1}]}(d\omega^{n-1}) \,
      \PV^n_x(d\omega^n) \, K^n(\omega^n, dx^{n+1}) \,
      \PV^{n+1}_{x^{n+1}}(d\omega^{n+1}) \cdots
 \end{align*}
with $\delta_{[\Delta^i]}$, being the Dirac-measure in $[\Delta^i]$, ensuring that $X$ starts $\PV_x$-a.s.\ in $E^n$.

Our main result on the concatenation of countably many right processes, which extends the concatenation of two processes given in \cite[Section~14]{Sharpe88}, is as follows:

\begin{theorem} \label{theo:concatenation countable}
 Let $(X^n, n \in \N)$ be a sequence of right processes on disjoint spaces $(E^n, n \in \N)$, such that the topological union
 $E := \bigcup_{n \in \N} E^n$ is a Radon space, and let a transfer kernel $K^n$ from $X^n$ to $(X^{n+1}, E^{n+1})$ be given for each $n \in \N$.
 Then the concatenation $X$ of the processes $(X^n, n \in \N)$ via the transfer kernels $(K^n, n \in \N)$
 is a right process on $E$. With $R^n := \inf \{ t \geq 0: X_t \in E^{n+1} \}$, for all $n \in \N$, $x \in \bigcup_{j=1}^n E^j$, $f \in b\sE^{n+1}$,
  \begin{align*}
   \EV_x \big( f(X_{R^n}) \, \1_{\{R^n < \infty \}} \, \big| \, \sF_{R^n-} \big) = K^n f \circ \pi^n \, \1_{\{R^n < \infty \}}. 
  \end{align*}
\end{theorem}

A standard method of constructing transfer kernels is by imposing conditional distributions $k^1(x, \,\cdot\,)$ for the 
transfer point (that is the ``revival point'' of $X^2$) given the ``exit point'' $X^1_{\zeta^1-} = x$ of $X^1$ (cf.~\cite[p.~78]{Sharpe88}):

\begin{example} \label{ex:revival kernel via killing distribution}
 Let $X^1$, $X^2$ be right processes on $E^1$, $E^2$ respectively,
 such that  $X^1_{\zeta^1-}$ exists a.s.\ in $E^1$,
 and let $k^1 \colon E^1 \times \sE^2 \rightarrow [0,1]$ be a probability kernel from $(E^1, \sE^1)$ to $(E^2, \sE^2)$.
 Then the map $K^1 \colon \Omega^1 \times \sE^2 \rightarrow [0,1]$ with
  \begin{align*}
   K^1(\omega^1, A) := k^1 \big( X^1_{\zeta^1-}(\omega^1), A \big), \quad \omega \in \Omega^1, A \in \sE^2,
  \end{align*}
 defines a transfer kernel from $X^1$ to $(X^2, E^2)$.
\end{example}

\subsection{Pasting of Two Processes: Construction Approach \& Main Result} \label{sec:intro pasting}

It is possible to weaken the assumption of disjoint subspaces $(E^n, n \in \N)$,
in order to apply the above described technique to paste together two right processes.
However, we then need to impose additional conditions on the subprocesses, namely,
they need to coincide on the shared state space, and their entry and exit distributions into this subset
must be equal irrespective of the mode of entry or exit (namely by either subprocess behavior or revival), see~figure~\ref{fig:pasting consistency}.

\begin{figure}[tb]
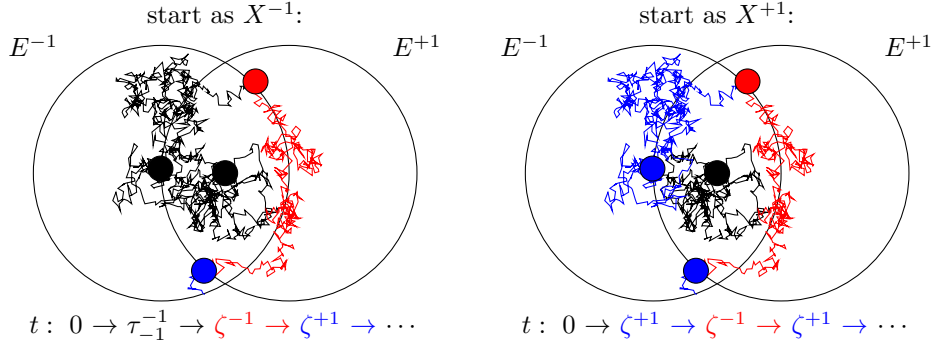
 
\begin{minipage}{\textwidth}
\begingroup
\parfillskip=0pt
\begin{minipage}[t]{0.483\textwidth}
  \centering
  start as $X^{-1}$:
  
  \input{C_CO_altern_A.tex}
  $t: ~ 0$ 
  $\rightarrow$ $\tau^{-1}_{-1}$ 
  $\rightarrow$ 
  {\color{red}$\zeta^{-1}$ $\rightarrow$}
  {\color{blue}$\zeta^{+1}$ $\rightarrow$} 
   $\cdots$
\end{minipage}%
\hfill
\begin{minipage}[t]{0.483\textwidth}
  \centering
  start as $X^{+1}$:
  
  \input{C_CO_altern_B.tex}
  $t: ~ 0$ 
  $\rightarrow$ 
  {\color{blue}$\zeta^{+1}$ $\rightarrow$}
  {\color{red}$\zeta^{-1}$ $\rightarrow$}
  {\color{blue}$\zeta^{+1}$ $\rightarrow$}
  $\cdots$
\end{minipage}%
\par\endgroup
\end{minipage}
\caption[Consistency condition for pasting two processes together]
        {Consistency condition for pasting together two processes $X^{-1}$, $X^{+1}$ on a common state space:
         The process behavior must be independent of the chosen starting process. 
         The left-hand picture shows a path behavior if the concatenated process is started as $X^{-1}$ (black), which is then revived after its death at $\zeta^{-1}$
         as $X^{+1}$ (red), afterwards revived as $X^{-1}$ at $\zeta^{+1}$ (blue), etc. 
        The concatenated process must show the same behavior if started as $X^{+1}$, as illustrated in the right-hand picture.} \label{fig:pasting consistency}
\end{figure}

Let $X^{-1}$, $X^{+1}$ be two right processes with lifetimes $\zeta^{-1}$, $\zeta^{+1}$ on $E^{-1}$, $E^{+1}$ respectively, 
and $K^{-1}$, $K^{+1}$ be transfer kernels from $X^{-1}$ to $(X^{+1}, E^{+1})$ and from $X^{+1}$ to $(X^{-1}, E^{-1})$. 
We define alternating copies of these processes and transfer kernels on disjoint state spaces by setting for each $n \in \N$
 \begin{align}\label{eq:pasting processes}
  X^n := \{n\} \times X^{(-1)^n}, \quad K^n := \delta_{n+1} \otimes K^{(-1)^n}.
 \end{align}
Then $X^n$ is a right process on $E^n := \{n\} \times E^{(-1)^n}$, $\sE^n = \{n\} \otimes \sE^{(-1)^n}$,
and $K^n$ is a transfer kernel from $X^n$ to $(X^{n+1}, E^{n+1})$.
Let $X$ be the concatenation of $(X^n, n \in \N)$ via the transfer kernels $(K^n, n \in \N)$. 
By Theorem~\ref{theo:concatenation countable}, it is a right process on $\tE = \bigcup_{n \in \N} E^n$,
equipped with the universal measurable sets $\tsE$.

Set $E := E^{-1} \cup E^{+1}$, and let $\pi \colon \tE \rightarrow E$ be the canonical projection onto the second coordinate.
The consistency conditions which ensure the pasted process~$\pi(X)$ to be a right process on $E$ are as follows:

\begin{theorem} \label{theo:alternating iterations}
 Let $X^{-1}$, $X^{+1}$ be right processes on spaces $E^{-1}$, $E^{+1}$ respectively, and $X$ be concatenation of $(X^n, n \in \N)$ via $(K^n, n \in \N)$, as defined 
 in~\eqref{eq:pasting processes}.
 Let $\tau^{-1}_{-1}$ be the first entry time of $X^{-1}$ into $E^{-1} \backslash E^{+1}$, 
 and $\tau^{+1}_{+1}$ be the first entry time of $X^{+1}$ into $E^{+1} \backslash E^{-1}$.
 If for all $x \in E^{-1} \cap E^{+1}$, $f \in b\sE$, $g^{-1} \in b\sE^{-1}$, $g^{+1} \in b\sE^{+1}$, the equalities
 \begin{enumerate}
  \item $   \EV^{-1}_x \big( \int_0^{\tau^{-1}_{-1}} e^{-\alpha t} \, f(X^{-1}_t) \, dt \big) 
          = \EV^{+1}_x \big( \int_0^{\tau^{+1}_{+1}} e^{-\alpha t} \, f(X^{+1}_t) \, dt \big) $,  \label{itm:alternating iterations  i}
  \item $   \EV^{-1}_x \big( e^{-\alpha \tau^{-1}_{-1}} \, g^{-1}(X^{-1}_{\tau^{-1}_{-1}}); \, \tau^{-1}_{-1} < \zeta^{-1} \big) 
          = \EV^{+1}_x \big( e^{-\alpha \zeta^{+1}} \, K^{+1} g^{-1};                      \, \zeta^{+1} < \tau^{+1}_{+1} \big) $, \\
        $   \EV^{+1}_x \big( e^{-\alpha \tau^{+1}_{+1}} \, g^{+1}(X^{+1}_{\tau^{+1}_{+1}}); \, \tau^{+1}_{+1} < \zeta^{+1} \big) 
          = \EV^{-1}_x \big( e^{-\alpha \zeta^{-1}} \, K^{-1} g^{+1};                      \, \zeta^{-1} < \tau^{-1}_{-1} \big) $  \label{itm:alternating iterations  ii}
 \end{enumerate}
 hold true, 
 then $\pi(X)$ is a right process on $E$, with $\pi \colon \tE \rightarrow E$ for $\tE = \bigcup_n \{n\} \times E^{(-1)^n}$, $E = E^{-1} \cup E^{+1}$.
\end{theorem}

The reader may observe that the second condition of the above theorem is not present in~\cite{Nagasawa76},
as Nagasawa only considers continuous processes with instant revivals at the exit points 
of the subprocesses.

If we only consider one process $X^0$ on $E$ and one transfer kernel $K^0$ from $X^0$ to~$(X^0, E)$,
and set $X^{-1} = X^{+1} = X^0$, $K^{-1} = K^{+1} = K^0$, no special conditions are required 
such that the pasted process $\pi(X)$ is a right process. 
We then obtain the following result for the instant revival process  (in the sense of \cite{IkedaNagasawaWatanabe66,Meyer75}),
constructed of copies of $X^0$ with the revival kernel $K^0$:

\begin{theorem} \label{theo:identical iterations}
 In the context of Theorem~\ref{theo:alternating iterations}, if $X^{-1} = X^{+1}$, $K^{-1} = K^{+1}$, then $\pi(X)$ is a right process on $E$.
\end{theorem}

\section{Concatenation of Right Processes}

 In this section, let $(X^n, n \in \N)$ be a sequence of right processes 
 \begin{align*}
  X^n = \big( \Omega^n, \sF^n, (\sF^n_t)_{t \geq 0}, (X^n_t)_{t \geq 0}, (\Theta^n_t)_{t \geq 0}, (\PV^n_x)_{x \in E^n} \big)
 \end{align*}
 on disjoint state spaces $(E^n, n \in \N)$,
 and for each $n \in \N$, let a transfer kernel $K^n$ from $X^n$ to $(X^{n+1}, E^{n+1})$ be given.
 The objective is to give a rigorous construction of the concatenation and to prove Theorem~\ref{theo:concatenation countable},
 which will be done incrementally by lifting the concatenation of finitely many processes to the countable case.  

\subsection{Concatenation of Two Processes} \label{sec:concatenation two}

Carrying out the specification given in section~\ref{sec:intro concatenation} for the case of two processes,
we set the concatenated process $X$ of $X^1$ and $X^2$ via the transfer kernel $K := K^1$ 
on the sample space 
 $\Omega := \Omega^1 \times \Omega^2$ with $\sigma$-algebra
 $\sF := \sF^1 \otimes \sF^2$
to be $X_t \colon \Omega \rightarrow E$, defined for each $t \geq 0$, $\omega = (\omega^1, \omega^2) \in \Omega$ by
 \begin{align*}
  X_t \big( (\omega^1, \omega^2) \big) :=
    \begin{cases}
      X^1_t (\omega^1),			 & t < \zeta^1(\omega^1), \\
      X^2_{t - \zeta^1(\omega^1)}(\omega^2),	 & t \geq \zeta^1(\omega^1),
    \end{cases}
 \end{align*}
 as well as introduce a family of operators $(\Theta_t, t \geq 0)$ on $\Omega$, defined by
  \begin{align*}
   \Theta_t \big( (\omega^1, \omega^2) \big) := 
     \begin{cases}
	\big( \Theta^1_t(\omega^1), \omega^2 \big),			   & t < \zeta^1(\omega^1), \\
	\big( [\Delta^1], \Theta^2_{t - \zeta^1(\omega^1)} (\omega^2) \big),  & t \geq \zeta^1(\omega^1).
     \end{cases} 
  \end{align*}
We use the transfer kernel $K$ to concatenate the processes $X^1$ and $X^2$ probabilistically by giving a transition between the
distributions $(\PV^1_x, x \in E^1)$ and $(\PV^2_x, x \in E^2)$. To this end, we define measures $(\PV_x, x \in E)$ on $\sF$ by
setting for $x \in E^1$, $H \in b(\sF^1 \otimes \sF^2)$:
 \begin{align*}
   \EV_x(H)
   & =
    \begin{cases}
     \int H(\omega^1, \omega^2) \, \PV^2_y(d\omega^2) \, K(\omega^1, dy) \, \PV^1_x(d\omega^1),   & x \in E^1, \\
     \int H(\omega^1, \omega^2) \, \PV^2_x(d\omega^2) \, \delta_{[\Delta^1]}(\omega^1),            & x \in E^2.
    \end{cases}   
 \end{align*}

The main result for the concatenation $X$ of two processes $X_1$ and $X_2$ via the transfer kernel $K$ is as follows: 
\begin{theorem} \label{theo:concatenation two}
 $X$ is a right process. For all $x \in E^1$, $f \in b\sE^2$, with the revival time $R := \inf \{ t \geq 0: X_t \in E^2 \}$,
  \begin{align*}
   \EV_x \big( f(X_{R}) \, \1_{\{R < \infty \}} \, \big| \, \sF_{R-} \big) = K f \circ \pi^1 \, \1_{\{R < \infty \}}.
  \end{align*}
\end{theorem}

This theorem is proved in detail in \cite[Theorem (14.8)]{Sharpe88} by an examination of the resolvent and of the excessive functions of 
the resulting concatenated process~$X$. We give a short sketch:

Using Dynkin's formula \eqref{eq:Dynkins formula (resolvent)} for decomposing the resolvent $(U_\alpha, \alpha > 0)$ of $X$ at the revival time $R$
(which a.s.\ coincides with the terminal time $\zeta^1$ of $X^1$), one obtains for $\alpha > 0$, $f \in b\cC(E)$, $x \in E = E^1 \cup E^2$,
 \begin{align*}
  U_\alpha f(x) 
  & = \1_{E^1}(x) \big( U^1_\alpha f^1(x) + \EV^1_x( e^{-\alpha \zeta^1} \, K U^2_\alpha f^2) \big) + \1_{E^2}(x) \, U^2_\alpha f^2(x),
 \end{align*}
with $f^j := \restr{f}{E^j}$, and $U^j$ being the resolvent of $X^j$, $j \in \{1,2\}$. 
An extensive analysis of the above components under the utilization of the strong Markov property of $X^1$ and $X^2$ as well as the properties 
of the transfer kernel $K$ then shows the Laplace-transformed equivalent of the Markov property for $X$. 
But $U^2_\alpha f^2$ is $\alpha$-excessive for $X^2$, 
and both $U^1_\alpha f^1$ and, by the shift properties of the transfer kernel $K$, the function $x \mapsto \EV^1_x( e^{-\alpha \zeta^1} \, K U^2_\alpha f^2)$
are $\alpha$-excessive for $X^1$. 
As both $X^1$ and $X^2$ satisfy~\eqref{eq:HD2}, it is immediate from the above decomposition that $t \mapsto U_\alpha f(X_t)$ is a.s.\ right continuous, 
which 
yields \eqref{eq:HD2} for $X$.
 
\subsection{Concatenation of Finitely Many Processes} \label{sec:concatenation finite}

Next, we consider for fixed $m \in \N$ the concatenation of the right processes $X^1, \ldots, X^m$ via the transfer kernels $K^1, \ldots, K^{m-1}$:
For every $n \in \{1, \ldots, m\}$
set $E^{(n)} := \bigcup_{j = 1}^n E^j$ as topological union of the spaces $(E^j, j \in \{1, \ldots, n\})$,
as well as $E := E^{(m)}$.
Directly extending the construction of section \ref{sec:concatenation two}, we define
the concatenated process $X$ on the sample space 
 $\Omega := \Omega^1 \times \cdots \times \Omega^m$ with $\sigma$-algebra
 $\sF := \sF^1 \otimes \cdots \otimes \sF^m$
to be $X_t \colon \Omega \rightarrow E$, defined for each $t \geq 0$, $\omega = (\omega^1, \ldots, \omega^m) \in \Omega$,
with
\begin{align}\label{eq:zeta_n nth killing time}
 \zeta^{(n)}(\omega) := \zeta^{(n)}(\omega^1, \ldots, \omega^{n}) := \zeta^1(\omega^1) + \cdots + \zeta^{n}(\omega^{n})
\end{align}
for $n \in \{1, \ldots, m-1\}$,
by
 \begin{align*}
  X_t(\omega)
 := \begin{cases}
      X^1_t(\omega^1),				 & t < \zeta^{(1)}(\omega), \\
      X^2_{t - \zeta^{(1)}(\omega)}(\omega^2),		 & \zeta^{(1)}(\omega) \leq t < \zeta^{(2)}(\omega), \\
      X^3_{t - \zeta^{(2)}(\omega)}(\omega^3),		 & \zeta^{(2)}(\omega) \leq t < \zeta^{(3)}(\omega), \\
      ~ \vdots					 & ~ \vdots \\
      X^m_{t - \zeta^{(m-1)}(\omega)}(\omega^m),  & t \geq \zeta^{(m-1)}(\omega),
    \end{cases}
 \end{align*}
Furthermore, we introduce a family of operators $(\Theta_t, t \geq 0)$ on $\Omega$ by setting for each $t \geq 0$, $\omega = (\omega^1, \ldots, \omega^m) \in \Omega$:
 \begin{align*}
    \Theta_t(\omega) :=  
     \begin{cases}
	\big( \Theta^1_t(\omega^1), \omega^2, \omega^3, \omega^4, \ldots, \omega^m \big),				 & t < \zeta^{(1)}(\omega), \\
	\big( [\Delta^1], \Theta^2_{t - \zeta^{(1)}(\omega)} (\omega^2), \omega^3, \omega^4, \ldots, \omega^m \big), 	 & \zeta^{(1)}(\omega) \leq t < \zeta^{(2)}(\omega), \\
	\big( [\Delta^1], [\Delta^2], \Theta^3_{t - \zeta^{(2)}(\omega)} (\omega^3), \omega^4, \ldots, \omega^m \big),		 & \zeta^{(2)}(\omega) \leq t < \zeta^{(3)}(\omega), \\
	~ \vdots & ~ \vdots \\
	\big( [\Delta^1], \ldots, [\Delta^{m-1}], \Theta^n_{t - \zeta^{(m-1)}(\omega)}(\omega^m) \big),  & t \geq \zeta^{(m-1)}(\omega),
     \end{cases}
 \end{align*}
The formal proof that $(\Theta_t, t \geq 0)$ is indeed a family of shift operators for $(X_t, t \geq 0)$ is a straight-forward computation with the help 
of the shift property \eqref{eq: lifetime shift} of the lifetime.
 
Like in the construction for two processes in above section \ref{sec:concatenation two}, we use the transfer kernels $(K^n, n \in \{1, \ldots, m-1\})$ 
to concatenate
the separate measures $(\PV^n_x, x \in E^n)$, $n \in \{1, \ldots, m\}$, of the partial processes $(X^n, n \in \{1, \ldots, m\})$.
For every $x \in E$, we define the measure $\PV_x$ on $\sF$ by setting for $x \in E^n$, $H \in b\sF$:
 \begin{align*} 
    \EV_x(H)   := \int H(\omega^1, \ldots, \omega^n) \, & \PV^n_{x^m}(d\omega^m) \, K^{m-1}(\omega^{m-1},dx^m) \, \PV^{m-1}_{x^{m-1}}(d\omega^{m-1}) \\
                                                &  \cdots ~ \PV^{n+1}_{x^{n+1}}(d\omega^{m+1}) \, K^{n}(\omega^{n},dx^{n+1}) \, \PV^{n}_{x}(d\omega^{n}) \\
                                                &  \delta_{[\Delta^{n-1}]}(d\omega^{n-1}) \cdots \delta_{[\Delta^{1}]}(d\omega^{1}).
 \end{align*}

Furthermore, we consider the $n$-th revival time
 \begin{align*}
  R^n := \inf \{ t \geq 0: X_t \in E^{n+1} \}, \quad n \in \{1, \ldots, m-1\},
 \end{align*}
which is terminal time, as $X$ is right continuous by construction, and every subspace $E^{n+1}$ is isolated in $E$.
 
The extension of Theorem~\ref{theo:concatenation two} to the finite concatenation $X$ of $X^1, \ldots, X^m$ via the transfer kernels $K^1, \ldots, K^{m-1}$ then reads as follows: 

\begin{theorem} \label{theo:concatenation finite}
 $X$ is a right process. For $n \in \{1, \ldots, m-1\}$, $x \in E^{(n)}$, $f \in b\sE^{n+1}$,
  \begin{align*}
   \EV_x \big( f(X_{R^n}) \, \1_{\{R^n < \infty \}} \, \big| \, \sF_{R^n-} \big) = K^n f \circ \pi^n \, \1_{\{R^n < \infty \}}.
  \end{align*}
\end{theorem}
 
We will prove this theorem iteratively,
that is, by assuming that 
 the concatenation $X^{(n)}$ of the processes $X^1, \ldots, X^n$ via the transfer kernels $K^1, \ldots, K^{n-1}$ is already a right process for any fixed $n \in \{1, \ldots, m-1\}$,
and then applying Sharpe's result (Theorem~\ref{theo:concatenation two}) in order to concatenate $X^{(n)}$ with $X^{n+1}$ via the transfer kernel $K^n$.
Before doing this, we need to lift the transfer kernels~$K^n$ from~$X^n$ (to~$(X^{n+1}, E^{n+1})$) to transfer kernels from~$X^{(n)}$ (to~$(X^{n+1}, E^{n+1})$).
We begin with a general result on stopping times:
\begin{lemma}\label{lem:random shift is measurable II} 
  Let $X$ be a right continuous strong Markov process, and $S$, $T$ be stopping times over the natural filtration $(\sF_t, t \geq 0)$,
  such that $S + T \circ \Theta_S = T$. Then $\Theta_S$ is $\sF_{T-}/\sF_{T-}$-measurable.
\end{lemma}
\begin{proof}
 It is well-known that $\Theta_S$ is $\sF_{t+S}/\sF_t$-measurable, see \cite[Corollary I.8.5]{BlumenthalGetoor69}.
 Consider the shift on a generating element of $\sF_{T-}$, that is for $t \geq 0$, $A \in \sF_t$,
  \begin{align*}
   \Theta_S^{-1} \big( A \cap \{ t < T \} \big)
   & = \Theta_S^{-1}(A) \cap \{ t < T \circ \Theta_S \} \\
   & = \Theta_S^{-1}(A) \cap \{ t + S < T \} \\
   & = \bigcup_{q \in \Q_+} \Big( \big( \Theta_S^{-1}(A) \cap \{ S < q - t \} \big) \cap \{ q < T \} \Big).
  \end{align*}
 As $\Theta_S^{-1}(A) \in \sF_{t + S}$, we see that, by the definition of $\sF_{t + S}$, the inner term satisfies
  \begin{align*}
   \forall q \in \Q_+: \quad \Theta_S^{-1}(A) \cap \{ t + S < q \} \in \sF_q.
  \end{align*}
 So every set of the countable union above is an element of~$\sF_{T-}$.
\end{proof}

In particular, the random times $S := \zeta^{(n-1)}$, $T := \zeta^{(n)}$ satisfy the requirements of the above lemma for the process $X^{(n)}$, 
in case it is a strongly Markovian.

\begin{lemma} \label{lem:lift of transfer kernel, finite}
 Assume $X^{(n)}$ is a strong Markov process for some $n \in \{1, \ldots, m-1\}$. Then $K^n \circ \pi^n$ defined by
  \begin{align*}
     K^n \circ \pi^n \big( (\omega^1, \ldots, \omega^n), dy \big) := K^n ( \omega^n, dy )
  \end{align*}
 is a transfer kernel from $X^{(n)}$ to $(X^{n+1}, E^{n+1})$.
\end{lemma}
\begin{proof}
 Obviously, $K^n \circ \pi^n$ is a probability measure in the second argument, because $K^n$ is a Markov kernel.
 In order to show the $\sF^{(n)}_{[\zeta^{(n)}-]}$-measurability of $K^n \circ \pi^n (\,\cdot\,, dy)$, 
 we start by observing that
  \begin{align*}
   \big( \pi^n \big)^{-1} \big( \sF^{n}_{\zeta^n-} \big)
   = \Omega^1 \times \cdots \times \Omega^{n-1} \times \sF^{n}_{\zeta^n-}
   \subseteq \sF^{(n)}_{\zeta^{(n)}-}.
  \end{align*}
 This can be seen by the following argument: The $\sigma$-algebra $\sF^{n}_{\zeta^n-}$ is generated by 
  \begin{align*}
   f(X^n_t) \, \1_{\{t < \zeta^n\}}, \quad f \in b\sE^n,
  \end{align*}
 and these functions, extended to $\Omega^{(n)}$, fulfill 
  \begin{align*}
     \big( f(X^n_t) \, \1_{\{t < \zeta^n\}} \big) \circ \pi^n 
    & = f \big( X^{(n)}_{t + \zeta^{(n-1)}} \big) \, \1_{\{t + \zeta^{(n-1)} < \zeta^{(n)}\}} \\
    & = \big( f(X^{(n)}_t) \, \1_{\{ t < \zeta^{(n)}\}} \big) \circ \Theta_{\zeta^{(n-1)}}.
  \end{align*}
 Because $X^{(n)}_t \, \1_{\{ t < \zeta^{(n)}\}}$ is $\sF^{(n)}_{\zeta^{(n)}-}$-measurable,
 Lemma~\ref{lem:random shift is measurable II} shows that the above function is indeed $\sF^{(n)}_{\zeta^{(n)}-}$-measurable.
 Therefore, we have $\big( \pi^n \big)^{-1} \big( \sF^{n}_{\zeta^n-} \big) \subseteq \sF^{(n)}_{\zeta^{(n)}-}$,
 and as $K^n (\,\cdot\,, dy)$ is $\sF^{n}_{[\zeta^n-]}$-measurable and $\pi^n$ is a projection, $K^n \circ \pi^n$ is $\sF^{(n)}_{\zeta^{(n)}-}$-measurable.
 
 It remains to prove that the shift invariance also lifts from $K^n$ to $K^n \circ \pi^n$:
 Fix $t \geq 0$ and let $N^n$ be a null set on $\sF^n$ such that, for all $\omega^n \in \complement N^n$,
  \begin{align*}
    K^n \circ \Theta^n_t (\omega^n) = K^n(\omega^n), \quad \text{if $t < \zeta^n(\omega^n)$}.
  \end{align*}
 But then $N^{(n)} := (\pi^n)^{-1} (N^n)$ is a null set on $\sF^{(n)}$, because
  \begin{align*}
   \PV^{(n)} \big( (\pi^n)^{-1} (N^n) \big) = \PV^n(N^n) = 0,
  \end{align*}
 and for all $\omega = (\omega^1, \ldots, \omega^n) \in \complement N^{(n)}$ (thus, $\omega^n \in \complement N^n$), we have for $t < \zeta^{(n)}(\omega)$:
  \begin{align*}
    (K^n \circ \pi^n) \circ \Theta^{(n)}_t (\omega) 
   & = \begin{cases}
        K^n (\omega^n), & t < \zeta^{(n-1)}(\omega), \\
        K^n \circ \Theta^n_{t - \zeta^{(n-1)}(\omega)}(\omega^n), & 0 \leq t - \zeta^{(n-1)}(\omega) <  \zeta^{n}(\omega)
       \end{cases} \\
   & = (K^n \circ \pi^n) (\omega),
  \end{align*}
 where we used the shift invariance of $K^n$ for the last identity.
\end{proof}

We are ready to prove the extension of Theorem~\ref{theo:concatenation two} to finitely many processes:

\begin{proof}[Proof of Theorem~\ref{theo:concatenation finite}]
 The case $m = 2$ is already proved, see Theorem~\ref{theo:concatenation two}.
 
 Assume now that, for some $m \in \N$, the process $X^{(m)}$ resulting from the concatenation of $X^1, \ldots, X^m$ via the transfer kernels $K^1, \ldots, K^{m-1}$
 is a right process and satisfies for all $n \in \{1, \ldots, m-1\}$, $x \in E^{(n)}$, $f \in b\sE^{n+1}$,
 with the revival time $R^{(n)} := \inf\{ t \geq 0: X^{(m)} \in E^{(n+1)} \}$ of $X^{(m)}$:
  \begin{align} \label{eq:concatenation finite, proof I}
   \EV_x \big( f(X^{(m)}_{R^{(n)}}) \, \1_{\{R^{(n)} < \infty \}} \, \big| \, \sF^{(m)}_{R^{(n)}-} \big) = K^n f \circ \pi^n \, \1_{\{R^{(n)} < \infty \}}.
  \end{align}
  
 Let $X^{(m+1)}$ be the concatenation of $X^{(m)}$ and $X^{m+1}$ via the transfer kernel $K^{(m)} := K^m \circ \pi^m$.
 By the pathwise definitions at the beginning of sections~\ref{sec:concatenation two} and~\ref{sec:concatenation finite},
 $X^{(m+1)}$ is equal to the process $X$ arising from the concatenation of
 $X^1, \ldots, X^m, X^{m+1}$ via the transfer kernels $K^1, \ldots, K^{m-1}, K^m$.
%
  In particular, the initial measures $\PV^{(m+1)}_x$, $\PV_x$ of $X^{(m+1)}$, $X$ respectively, coincide for all $x \in E^{(m+1)}$.
  
  Now Theorem~\ref{theo:concatenation two} states that $X = X^{(m+1)}$ is a right process, 
  and that,
  with the revival time $R^m = \inf \{ t \geq 0: X_t \in E^{m+1} \} =: R^{(m)}$,
  it satisfies, with the projection $\pi^{(m)} \colon \Omega \rightarrow \Omega^1 \times \cdots \times \Omega^m$ to the first $m$ coordinates:
     \begin{align*}
       \EV_x \big( f(X_{R^m}) \, \1_{\{R^m < \infty \}} \, \big| \, \sF_{R^m-} \big)
       & = \EV^{(m+1)}_x \big( f(X^{(m+1)}_{R^{(m)}}) \, \1_{\{R^{(m)} < \infty \}} \, \big| \, \sF^{(m+1)}_{R^{(m)}-} \big) \\
       & = (K^m \circ \pi^m) f \circ \pi^{(m)} \, \1_{\{R^{(m)} < \infty \}} \\
       & = (K^m  f) \circ \pi^{m} \, \1_{\{R^m < \infty \}}.
     \end{align*}
  Assumption \eqref{eq:concatenation finite, proof I} for $X^{(m)}$ concludes the proof, as we get for $n \in \{1, \ldots, m-1\}$:
     \begin{align*}
       (K^n  f) \circ \pi^{n} \, \1_{\{R^n < \infty \}}
       & = \EV^{(m)}_x \big( f(X^{(m)}_{R^{(n)}}) \, \1_{\{R^{(n)} < \infty \}} \, \big| \, \sF^{(m)}_{R^{(n)}-} \big) \circ \pi^{(m)} \\
       & = \EV_x \big( f(X_{R^n}) \, \1_{\{R^n < \infty \}} \, \big| \, \sF_{R^n-} \big).
     \end{align*}   
  Here, the equality of both conditional expectations is seen as follows:
  Because $R^n = R^{(n)} \circ \pi^{(m)}$ and $X_t = X^{(m)}_t \circ \pi^{(m)}$ hold for all $t < R^{(m)}$,
  we have the relation $X_{R^n} = X^{(m)}_{R^n} \circ \pi^{(m)}$. The $\sigma$-algebras $\sF_{R^n-}$ and $\sF^{(m)}_{R^{(n)}-}$
  are generated by the multiplicatively closed classes of functions
   \begin{align*}
    J & := f_1(X_{t_1}) \cdots f_k(X_{t_k}) \, \1_{ \{ t < R^n \} },  \\
    J^{(m)} &:= f_1(X^{(m)}_{t_1}) \cdots f_k(X^{(m)}_{t_k}) \, \1_{ \{ t < R^{(n)} \} },
   \end{align*}
  with $0 \leq t_1 < \cdots < t_k \leq t$, $f_1, \ldots, f_k \in b\sE$. 
  It is immediate that $J = J^{(m)} \circ \pi^{(m)}$.
  Therefore, the integrals of both functions are the same (over their respective spaces), that~is, we obtain 
  \begin{align*}
    \EV_x \big( f(X_{R^m}) \, \1_{\{R^m < \infty \}} \, J \big)
    & = \EV^{(m)}_x \big( f(X_{R^{(m)}}) \, \1_{\{R^{(m)} < \infty \}} \, J^{(m)} \big) \\
    & = \EV^{(m)}_x \big( \EV^{(m)}_x \big( f(X^{(m)}_{R^{(n)}}) \, \1_{\{R^{(n)} < \infty \}} \, \big| \, \sF^{(m)}_{R^{(n)}-} \big) \, J^{(m)} \big) \\
    & = \EV_x \big( \EV^{(m)}_x \big( f(X^{(m)}_{R^{(n)}}) \, \1_{\{R^{(n)} < \infty \}} \, \big| \, \sF^{(m)}_{R^{(n)}-} \big) \circ \pi^{(m)} \, J \big).
  \end{align*}  
  On the other hand,  $\pi^{(m)}$ is $\sF_{R^n-}/\sF^{(m)}_{R^{(m)}-}$-measurable, because for all $f \in b\sE$, 
   \begin{align*}
     f(X_t) \, \1_{\{t < R^n\}} = f(X^{(m)}_t) \, \1_{\{t < R^{(n)}\}} \circ \pi^{(m)},
   \end{align*}
  which yields the $\sF_{R^n-}$-measurability of $\EV^{(m)}_x \big( f(X^{(m)}_{R^{(n)}}) \, \1_{\{R^{(n)} < \infty \}} \, \big| \, \sF^{(m)}_{R^{(n)}-} \big) \circ \pi^{(m)}$.
\end{proof}

\subsection{Concatenation of Countably Many Processes}

We are ready to turn to the concatenation of the processes $(X^n, n \in \N)$ via the transfer kernels $(K^n, n \in \N)$:
We assume the topological union $E = \bigcup_{n \in \N} E^n$  of the disjoint spaces $(E^n, n \in \N)$ to be a Radon space.
For instance, this is the case if the spaces $E^n$, $n \in \N$, are Lusin,
see \cite[Corollary to Lemma~II.5]{Schwartz73}. Adjoin a point $\Delta \notin E$ as a new, isolated point and form $E_\Delta := E \cup \{\Delta\}$.

Following the construction of section~\ref{sec:concatenation finite}, let $\zeta^{(n)}$ be given as in equation~\eqref{eq:zeta_n nth killing time} for each $n \in \N$.
We define the process $X_t \colon \Omega \rightarrow E_\Delta$ and
the family of shift operators $(\Theta_t, t \geq 0)$ for $X$ 
on $\Omega := \prod_{n \in \N} \Omega^n$ by setting for all $t \geq 0$, $\omega = (\omega^1, \omega^2, \ldots) \in \Omega$ with
$\zeta^{(n-1)}(\omega) \leq t < \zeta^{(n)}(\omega)$, $n \in \N$,
 \begin{align*}
  X_t (\omega) & := X^n_{t -\zeta^{(n-1)}(\omega)}(\omega^n), \\
  \Theta_t (\omega) & := \big( [\Delta^1], \ldots, [\Delta^{n-1}], \Theta^n_{t - \zeta^{(n-1)}(\omega)}(\omega^n), \omega^{n+1}, \omega^{n+2}, \ldots \big),
 \end{align*}
as well as $X_t(\omega) := \Delta$ and $\Theta_t(\omega) := \big( [\Delta^1], [\Delta^2], [\Delta^3], \ldots \big)$ for all $t \geq \sum_{n \in \N} \zeta^n(\omega^n)$.
The right continuity of all underlying processes $X^n$, $n \in \N$, yields the right continuity of~$X$.

Set $\sF := \bigotimes_{n \in \N} \sF^n$, and introduce the measures $(\PV_x, x \in E)$ on $(\Omega, \sF)$
by constituting a transition between the
subprocesses' distributions $(\PV^n_x, x \in E^n)$, $n \in \N$, via the transfer kernels $(K^n, n \in \N)$.
To this end, we define the probability measures $(\PV_x, x \in E)$ as projective limits of the following prescriptions:
For any $m \in \N$ and $H \in b(\sF^1 \otimes \cdots \otimes \sF^m)$, we set for $x \in E^1$
 \begin{align*}
    \EV_x(H)  := \int H(\omega^1, \ldots, \omega^m) \, &  \PV^m_{x^m}(d\omega^m) \, K^{m-1}(\omega^{m-1},dx^m) \, \PV^{m-1}_{x^{m-1}}(d\omega^{m-1})  \\
                                               &  \cdots ~ \PV^2_{x^2}(d\omega^2) \, K^{1}(\omega^{1},dx^2) \, \PV^{  1}_{x}(d\omega^{1}),
 \end{align*}
while for $x \in E^n$, $n \geq 2$, we set 
 \begin{align*} 
    \EV_x(H)  := \int H(\omega^1, \ldots, \omega^m) \, & \PV^m_{x^m}(d\omega^m) \, K^{m-1}(\omega^{m-1},dx^m) \, \PV^{m-1}_{x^{m-1}}(d\omega^{m-1}) \\
                               &    \cdots ~ \PV^{n+1}_{x^{n+1}}(d\omega^{n+1}) \, K^{n}(\omega^{n},dx^{n+1}) \, \PV^{n}_{x}(d\omega^{n}) \\
                               &    \delta_{[\Delta^{n-1}]}(d\omega^{n-1}) \cdots \delta_{[\Delta^{1}]}(d\omega^{1}).
 \end{align*}
An easy calculation shows that the above definitions admit consistency and therefore, by the Kolmogorov existence theorem,
exist as measures on $(\Omega, \sF)$. 

We are going to prepare the main method for the proof that $X$ is a right process.
A stability result for right processes, which will be made rigorous in Lemma~\ref{lem:right process if subprocesses are right} below, 
states the following: Assume we are given a stochastic process $X$ and an increasing sequence of terminal times $(R^n, n \in \N)$.
If process $X$ killed at $R^n$ is a right process for every $n \in \N$, then $X$ killed at $R := \lim_n R^n$ is a right process as well.
This result is then directly applicable in our context, 
because, for every $n \in \N$, the concatenated process $X$ killed at the $n$-th revival time $R^n$ 
  \begin{align*}
   R^n & :=  \inf\Big\{ t \geq 0: ~ X_t \in \bigcup_{m = n+1}^\infty E^{m} \Big\} = \zeta^{(n)} \\
       & \ =   \inf\big\{ t \geq 0: ~ X_t \in E^{n+1} \big\} \qquad \text{$\PV_x$-a.s.\ for $\textstyle x \in \bigcup_{m \leq n} E^m$}
  \end{align*}
is just the finite concatenation
of $X^1, \ldots, X^n$ via $K^1, \ldots, K^{n-1}$, which is a right process by the results of section~\ref{sec:concatenation finite}.
Thus, $X$ killed at $\lim_n R^n = \sum_n \zeta^n$ (which equals $X$ by construction) is proved to be a right process.

\begin{lemma} \label{lem:right process if subprocesses are right}
 Let $(X_t, t \geq 0)$ be a right continuous stochastic process on a measurable space $(\Omega, \sF)$ with values in a Radon space~$E$,
 $(\PV_x, x \in E)$ be a family of probability measures on a measurable space $(\Omega, \sF)$,
 $(R^n, n \in \N)$ be an increasing sequence of random times with $R := \sup_{n \in \N} R^n$, 
 and $(E^{R,n}, n \in \N)$ be an increasing sequence of Radon spaces.
 Define the processes $(X^{R,n}_t, t \geq 0)$, $n \in \N$, and $(X^R_t, t \geq 0)$ on~$\Omega$ by 
  \begin{align*}
   X^{R,n}_t = \begin{cases}
                X_t, & t < R^n, \\
                \Delta,  & t \geq R^n,
               \end{cases}
  \quad \text{and} \quad
   X^R_t = \begin{cases}
                X_t, & t < R, \\
                \Delta,  & t \geq R,
             \end{cases}   
   \quad t \geq 0.
  \end{align*}
 Then $X^R = \big( \Omega, \sF, (\sF^R_t)_{t \geq 0}, (X^R_t)_{t \geq 0}, (\Theta^R_t)_{t \geq 0}, (\PV_x)_{x \in E} \big)$,
 with $(\sF^R_t, t \geq 0)$ being the natural filtration of~$X^{R}$ and $(\Theta^R_t, t \geq 0)$ being an arbitrary family of shift operators for~$X$,
 is a right process on~$E$, if the following conditions are fulfilled:
 \begin{enumerate}
  \item $(R^n, n \in \N)$ is a sequence of stopping times over $(\sF^R_t, t \geq 0)$; 
  \item $(E^{R,n}, n \in \N)$ increases to~$E$, that is, $\bigcup_{n \in \N} E^{R,n} = E$;
  \item for each $n \in \N$, there exist a filtration $(\sF^{R,n}_t, t \geq 0)$ on $(\Omega, \sF)$ and 
        a family of operators $(\Theta^{R,n}_t, t \geq 0)$ on~$\Omega$, such that
         \begin{align*}
          X^{R,n} := \big( \Omega, \sF, (\sF^{R,n}_t)_{t \geq 0}, (X^{R,n}_t)_{t \geq 0}, (\Theta^{R,n}_t)_{t \geq 0}, (\PV_x)_{x \in E^{R,n}} \big)
         \end{align*}
        is a right process on~$E^{R,n}$;
  \item for each $n \in \N$, $R^n$ is a terminal time for the process $X^{R,n}$,
        satisfying $R^n > 0$ $\PV_x$-a.s.\ for all $x \in E^{R,n}$.
 \end{enumerate}
 \end{lemma}
\begin{proof} \
 The process $X^R$ is normal, because for any $x \in E$, with $n \in \N$ such that $x \in E^{R,n}$, the normality of $X^{R,n}$ gives
  \begin{align*}
   \PV_x( X^R_0 = x ) = \PV_x ( X^{R,n}_0 = x, R^n > 0 ) = 1.
  \end{align*}
  
 Turning to the Markov property of $X^R$,
    let $s, t \geq 0$ and $f \in b\sE$.
    For any $k \in \N$, $0 = t_0 < t_1 < t_2 < \cdots < t_k \leq t$, $g_0 \in b\sE$, $g_1, \ldots, g_k \in b\sE$, set
     \begin{align*}
      J^R & := g_0(X^R_{t_0}) \, g_1(X^R_{t_1}) \cdots g_k(X^R_{t_k}), \\
      J^{R,n} & := g_0(X^{R,n}_{t_0}) \, g_1(X^{R,n}_{t_1}) \cdots g_k(X^{R,n}_{t_k}), \quad n \in \N. 
     \end{align*}
  As the set of functions of the type $J^R$ forms a multiplicatively closed generator of~$b\sF^R_t$,
  and as $\EV_{X^R_t} \big( f(X^R_{s}) \big)$ is measurable with respect to the natural filtration $(\sF^R_t, t \geq 0)$, it suffices to show that 
   \begin{align*}
    \EV_x \big( f(X^R_{s+t}) \cdot J^R \big) = \EV_x \Big( \EV_{X^R_t} \big( f(X^R_{s}) \big) \cdot J^R \Big).
   \end{align*}
  We start by observing that $\{s+t < R\} = \bigcup_n \{s+t < R^n\}$ and $X^R_{s+t} = X^{R,n}_{s+t}$ on~$\{ s+t < R^n \}$, 
  so Lebesgue's dominated convergence theorem yields
     \begin{align*}
       \EV_x \Big( f(X^R_{s+t}) \cdot J^R \Big) 
         & = \lim_n \EV_x \Big( f(X^{R,n}_{s+t}) \cdot J^{R,n}; \, s+t < R^n \Big).
     \end{align*}   
    By employing both
    the terminal time property and the stopping time property of~$R^n$ with respect to~$X^{R,n}$ next, we obtain
    \begin{align*}
         &   \lim_n \EV_x \Big( f(X^{R,n}_{s+t}) \cdot J^{R,n}; \, s+t < R^n \Big) \\
         & = \lim_n \EV_x \Big( f(X^{R,n}_{s}) \circ \Theta^{R,n}_t \cdot J^{R,n}; \, s < R^{R,n} \circ \Theta^{R,n}_t, t < R^n \Big) \\
         & = \lim_n \EV_x \Big( \EV_x \big( f(X^{R,n}_{s}) \circ \Theta^{R,n}_t ; \, s < R^n \circ \Theta^{R,n}_t \, \big| \, \sF^{R,n}_t \big) \cdot J^{R,n} ; \, t < R^n \Big).
    \end{align*}    
    Now, we are able to apply the Markov property of $X^{R,n}$, which yields
    \begin{align*}
         & \lim_n \EV_x \Big( \EV_x \big( f(X^{R,n}_{s}) \circ \Theta^{R,n}_t ; \, s < R^n \circ \Theta^{R,n}_t \, \big| \, \sF^{R,n}_t \big) \cdot J^{R,n} ; \, t < R^n \Big) \\
         & = \lim_n \EV_x \Big( \EV_{X^{R,n}_t} \big( f(X^{R,n}_{s}) ; \, s < R^n \big) \cdot J^{R,n}; \, t < R^n \Big),
    \end{align*}     
    and by carrying out the above steps in reverse order, we conclude that    
    \begin{align*}
         &   \lim_n \EV_x \Big( \EV_{X^{R,n}_t} \big( f(X^{R,n}_{s}) ; \, s < R^n \big) \cdot J^{R,n}; \, t < R^n \Big) \\
         & = \EV_x \Big( \EV_{X^R_t} \big( f(X^R_{s}) ; \, s < R \big) \cdot J^R; \, t < R \Big) \\
         & = \EV_x \Big( \EV_{X^R_t} \big( f(X^R_{s}) \big) \cdot J^R \Big).
     \end{align*}
     
    It remains to verify that $t \mapsto f(X^R_t)$ is a.s.\ right continuous for all $\alpha$-excessive functions~$f$.
    To this end, let $\sS_\alpha(X^{R,n})$, $\sS_\alpha(X^R)$, $\alpha > 0$, be the sets of all $\alpha$-excessive functions,
    $T^{n}_t$, $T^{R}_t$, $t \geq 0$, be the transition operators, and
    $U^{n}_\alpha$, $U^{R}_\alpha$, $\alpha > 0$, be the $\alpha$-potential operators 
    of the processes $X^{R,n}$, $X^R$ respectively, that is,
      \begin{align*}
       U^{n}_\alpha h (x) = \EV_x \Big( \int_0^\infty e^{-\alpha s} \, h(X^{R,n}_s) \, ds \Big), \quad h \in p\sE, n \in \N.
      \end{align*}
    Now let $f \in \sS_\alpha(X^R)$. Then there exists a sequence $(h_m, m \in \N)$ in $bp\sE$ such that
      \begin{align*}
       f = \sup_m U^{R}_\alpha h_m.
      \end{align*}
    Of course, $U^{R}_\alpha h_m$ is in $\sS_\alpha(X^R)$ (see, e.g., \cite[Proposition 2.2]{ChungWalsh05}).
    However, we are going to prove now that this potential, as a function restricted to~$E^{R,n}$, is also in~$\sS_\alpha(X^{R,n})$.
    As~$X^{R,n}$ is a subprocess of $X^R$, we have
      \begin{align*}
       e^{-\alpha t} \, T^{n}_t \, U^{R}_\alpha h_m
         & = \EV \Big( e^{-\alpha t} \, U^{R}_\alpha h_m(X^{R,n}_t) \Big) \\
         & = \EV \Big( e^{-\alpha t} \, U^{R}_\alpha h_m(X^R_t) ; \, t < R^n \Big) \\
         & = \EV \Big( e^{-\alpha t} \, \EV_{X^R_t} \Big( \int_0^\infty e^{-\alpha s} \, h_m(X^R_s) \, ds \Big) ; \, t < R^n \Big).
      \end{align*}  
    The Markov property of $X^R$ and the stopping time property of~$R^n$ with respect to~$X^{R}$ imply that this is equal to
      \begin{align*}
       e^{-\alpha t} \, T^{n}_t \, U^{R}_\alpha h_m
         & = \EV \Big( \EV \Big( \int_t^\infty e^{-\alpha s} \, h_m(X^R_s) \, ds \, \big| \, \sF^R_t \Big) ; \, t < R^n \Big) \\
         & = \EV \Big( \int_t^\infty e^{-\alpha s} \, h_m(X^R_s) \, ds ; \, t < R^n \Big).
      \end{align*}   
    Therefore, we have $e^{-\alpha t} \, T^{n}_t \, U^{R}_\alpha h_m \leq U^{R}_\alpha h_m$ for all $t \geq 0$, and
    because $R^n > 0$ holds $\PV_x$-a.s.\ for all $x \in E^{R,n}$, Levi's monotone convergence theorem yields
     \begin{align*}
      \lim_{t \downarrow 0} e^{-\alpha t} \, T^{n}_t \, U^{R}_\alpha h_m 
        & = \EV \Big( \int_0^\infty e^{-\alpha s} \, h_m(X^R_s) \, ds \Big) \\
        & = U^{R}_\alpha h_m
     \end{align*}
    on $E^{R,n}$. Thus $\restr{U^{R}_\alpha h_m}{E^{R,n}} \in \sS^\alpha(X^{R,n})$ for each $n \in \N$, and as the set of excessive functions is closed under suprema, we have
     \begin{align*}
      \restr{f}{E^{R,n}} = \sup_m \left( \restr{U^{R}_\alpha h_m}{E^{R,n}} \right) \in \sS_\alpha(X^{R,n}).
     \end{align*}

    We are now able to conclude that $X$ satisfies \eqref{eq:HD2}:
    We have just seen that, for any $f \in \sS_\alpha(X^R)$, $f$ restricted on $E^{R,n}$ is $\alpha$-excessive for $X^{R,n}$ for all $n \in \N$,
    so as $X^{R,n}$ is a right process, the map $t \mapsto f(X^{R,n}_t)$ is a.s.\ right continuous for each $n \in \N$.
    With $X^R_t = X^{R,n}_t$ on $t < R^n$, $\lim_n R^n = R$ and $f(\Delta) = 0$, we immediately get that $t \mapsto f(X^R_t)$ is a.s.\ right continuous.    
\end{proof}

Let $X$ be the concatenation of the right processes $(X^n, n \in \N)$ via the transfer kernels $(K^n, n \in \N)$, as constructed above,
and $(R^n, n \in \N)$ be the revival times of $X$.
As announced, we are going to apply Lemma~\ref{lem:right process if subprocesses are right} with $X^{R,n}$ being the subprocesses
of~$X$ killed at the revival times $R^n$, that is, we consider for all $\omega = (\omega^1, \omega^2, \ldots) \in \Omega$, $t \geq 0$, 
 \begin{equation} \label{eq:def X_R_n}
 \begin{aligned}
     X^{R,n}_t(\omega)
  & := \begin{cases}
      X_t (\omega), & t < R^n, \\
      \Delta,  & t \geq R^n
    \end{cases} \\
  & = \begin{cases}
      X^1_t(\omega^1),				 & t < \zeta^{(1)}(\omega), \\
      X^2_{t - \zeta^{(1)}(\omega)}(\omega^2),		 & \zeta^{(1)}(\omega)) \leq t < \zeta^{(2)}(\omega), \\
      ~ \vdots					 & ~ \vdots \\
      X^n_{t - \zeta^{(n-1)}(\omega)}(\omega^n),  & \zeta^{(n-1)}(\omega) \leq t < \zeta^{(n)}(\omega) \\
      \Delta,                                        & t \geq \zeta^{(n)}(\omega),
    \end{cases}
 \end{aligned}
 \end{equation}
equipped with shift operators $(\Theta^{R,n}_t, t \geq 0)$ defined by
 \begin{align*}
   & \Theta^{R,n}_t(\omega) :=  \\
   &  \begin{cases}
	\big( \Theta^1_t(\omega^1), \omega^2, \ldots \big),				 & t < \zeta^{(1)}(\omega), \\
	\big( [\Delta^1], \Theta^2_{t - \zeta^{(1)}(\omega)} (\omega^2), \omega^3, \ldots \big), 	 & \zeta^{(1)}(\omega) \leq t < \zeta^{(2)}(\omega), \\
	~ \vdots & ~ \vdots \\
	\big( [\Delta^1], \ldots, [\Delta^{n-1}], \Theta^n_{t - \zeta^{(n-1)}(\omega)}(\omega^n), \omega^{n+1}, \ldots \big),  & \zeta^{(n-1)}(\omega) \leq t < \zeta^{(n)}(\omega) \\
	\big( [\Delta^1], \ldots, [\Delta^{n-1}], [\Delta^{n}], \omega^{n+1}, \ldots \big),                      & t \geq \zeta^{(n)}(\omega).                                                                                                                   
     \end{cases}
 \end{align*}

We first need to show that the subprocesses $X^{R,n}$, $n \in \N$, fulfill the requirements of Lemma~\ref{lem:right process if subprocesses are right}.
In particular, they are right processes:
 
\begin{lemma} \label{lem:concatenation count, subprocesses are right}
 For every $n \in \N$, the process
  \begin{align*}
   X^{R,n} = \big( \Omega, \sF, (\sF^{R,n}_t)_{t \geq 0}, (X^{R,n}_t)_{t \geq 0}, (\Theta^{R,n}_t)_{t \geq 0}, (\PV_x)_{x \in E^{R,n}} \big), 
  \end{align*}
 with $(\sF^{R,n}_t, t \geq 0)$ being its natural filtration,
 is a right process on the state space $E^{(n)} := \bigcup_{j=1}^n E^j$.
\end{lemma}
\begin{proof}
 Consider $X^{(n)} = \big( \Omega^{(n)}, \sF^{(n)}, (\sF^{(n)}_t)_{t \geq 0}, (X^{(n)}_t)_{t \geq 0}, (\Theta^{(n)}_t)_{t \geq 0}, (\PV^{(n)}_x)_{x \in E^{(n)}} \big)$
 the concatenation of $X^1, \ldots, X^n$ with the transfer kernels $K^1, \ldots, K^{n-1}$. 
 Then $X^{(n)}$ is a right process on $E^{(n)}$ by Theorem~\ref{theo:concatenation finite}.
 
 Let $\pi^{(n)} \colon \Omega \rightarrow \Omega^{(n)}$ be the canonical projection onto $\Omega^{(n)}$.
 By checking the decomposition~\eqref{eq:def X_R_n} and the definition of $X^{(n)}$ in section~\ref{sec:concatenation finite},
 it is evident that
  \begin{align*} X^{R,n}_t = X^{(n)}_t \circ \pi^{(n)} \quad \text{for all $t \geq 0$, a.s.\ on $\Omega$.} \end{align*}
 The definitions of the measures $\PV_x$, $\PV_x^{(n)}$ for the countable and finite concatenations yield that for all $x \in E^{(n)}$,
  \begin{align*} \PV_x \circ ( \pi^{(n)} )^{-1} = \PV^{(n)}_x \quad \text{on $\sF^{(n)} = \sF^1 \otimes \cdots \otimes \sF^n$}. \end{align*}
 Thus, $X^{R,n}$ and $X^{(n)}$ have the same finite dimensional distributions (with respect to their corresponding measures $\PV$ and $\PV^{(n)}$):
  \begin{align} \label{eq:concatenation count, subprocesses are right, proof I} 
   \PV_x \circ \big( X^{R,n}_{t_1}, \ldots, X^{R,n}_{t_k} \big)^{-1} = \PV^{(n)}_x \circ \big( X^{(n)}_{t_1}, \ldots, X^{(n)}_{t_k} \big)^{-1}.
  \end{align}  
 This easily transfers the normality and Markov property from $X^{(n)}$ to $X^{R,n}$.
 Turning to \eqref{eq:HD2} for $X^{R,n}$, we observe that the $\alpha$-excessive functions of $X^{(n)}$ and $X^{R,n}$ coincide, 
 as the transition operators $T^{(n)}_t$, $T^{R,n}_t$, $t \geq 0$, of $X^{(n)}$, $X^{R,n}$ agree for all $f \in p\sE^{(n)}$, $x \in E^{(n)}$:
      \begin{align*} 
       T^{R,n}_t f(x) = \EV_x \big( f(X^{R,n}_t) \big) = \EV^{(n)}_x \big( f(X^{(n)}_t) \big) = T^{(n)}_t f(x).
     \end{align*}    
 But $X^{(n)}$ is a right process, so for any $f \in \sS_\alpha(X^{R,n})$,
    \begin{align*} 
        t \mapsto f \big( X^{R,n}_t \big) = f \big( X^{(n)}_t \circ \pi^{(n)} \big)
    \end{align*}
  is a.s.\ right continuous, as for any $\PV_x^{(n)}$-null~set $N$ in $\sF^{(n)}$, $(\pi^{(n)})^{-1}(N)$ is a $\PV_x$-null~set in~$\sF$.
\end{proof}

We are now able to use Lemma~\ref{lem:concatenation count, subprocesses are right} to lift Theorem~\ref{theo:concatenation finite}
to the concatenation of countably many processes:

\begin{proof}[Proof of Theorem~\ref{theo:concatenation countable}]
 Let $X^{R,n}$ be the processes as defined above Lemma~\ref{lem:concatenation count, subprocesses are right} for the revival times  
 $R^n$, $n \in \N$, equipped with their natural filtrations, on their state spaces $E^{R,n} := E^{(n)}$.
 Then the sequence $(R^n, n \in \N)$ increases to the lifetime of~$X$, and the sequence $(E^{R,n}, n \in \N)$ increases to $E = \bigcup_n E^n$.
 Furthermore, by Lemma~\ref{lem:concatenation count, subprocesses are right}, the process $X^{R,n}$ is a right processes on $E^{R,n}$ for every $n \in \N$,
 and being a subprocess of $X$, its natural filtration satisfies $\sF^{R,n} \subseteq \sF^R$. Finally, $R^n$ coincides with its lifetime, so it is a terminal time for~$X^{R,n}$,
 and being the first entry time of $X$ into a closed set, it is also a stopping time for~$X$.
 Thus, Lemma~\ref{lem:right process if subprocesses are right} is applicable, which shows that $X = X^R$ is a right process.
 
 It only remains to prove the revival formula given in Theorem~\ref{theo:concatenation countable}.
 To this end, we compare once again the processes $X^{R,n}$ and $X^{(n)}$ like in the proof of Lemma~\ref{lem:concatenation count, subprocesses are right}:
 
 As $X^{(n+1)}$ is the concatenation of $X^{(n)}$ and $X^{n+1}$ with transfer kernel $K^n \circ \pi^n$ (see section~\ref{sec:concatenation finite}), 
 Theorem~\ref{theo:concatenation finite} yields, with $R^{(n)} = \inf \{ t \geq 0: X^{(n+1)}_t \in E^{n+1} \}$:
  \begin{align*} 
    \EV^{(n+1)}_x \big( f(X^{(n+1)}_{R^{(n)}}) \, \1_{\{R^{(n)} < \infty \}} \, \big| \, \sF^{(n+1)}_{R^{(n)}-} \big) = K^n f \circ \pi^n \, \1_{\{R^{(n)} < \infty \}}.
  \end{align*}
 Checking the construction of $X$ and $X^{(n+1)}$, we observe that
  \begin{align*} R^{(n)} \circ \pi^{(n+1)} = R^n  \quad \text{and} \quad   X^{(n+1)}_{R^{(n)}} \circ \pi^{(n+1)} = X^{R,n+1}_{R^n}  \quad \text{a.s.\ on $\Omega$}. \end{align*}
 
 By definition, $\sF_{R^n-} = \sigma \big( \big\{ A \cap \{t < R^n\} : t \geq 0, A \in \sF_t \big\} \big)$, and this generator is $\cap$-stable, because
  for all $s, t \geq 0$, $A_s \in \sF_s$, $A_t \in \sF_t$, with $s \leq t$:
   \begin{align*} \big( A_s \cap \{s < R^n\} \big) \cap  \big( A_t \cap \{t < R^n\} \big) = \big( A_s \cap A_t \big) \cap \{t < R^n\}, \end{align*}
  and $A_s \cap A_t \in \sF_t$.
 Thus, it suffices to show that for all $t \geq 0$, $f \in b\sE$, $k \in \N$, $0 \leq t_1 < \cdots < t_k \leq t$, $g_1, \ldots, g_k \in b\sE$ with 
   \begin{align*}
     J         & := g_1(X_{t_1}) \cdots g_k(X_{t_k}) \cdot \1_{\{ t < R^n \}}, \\
     J^{R,n+1}   & := g_1(X^{R,n+1}_{t_1}) \cdots g_k(X^{R,n+1}_{t_k}) \cdot \1_{\{ t < R^n \}}, \\
     J^{(n+1)} & := g_1(X^{(n+1)}_{t_1}) \cdots g_k(X^{(n+1)}_{t_k}) \cdot \1_{\{ t < R^{(n)} \}}
   \end{align*}
 the following holds true, as $X_{R^n} = X^{R,n+1}_{R^n}$ a.s.:
   \begin{align*}
    \EV_x \big( f(X_{R^n}) \, \1_{\{R^n < \infty \}} \cdot J \big) 
    & = \EV_x \big( f(X^{R,n+1}_{R^n}) \, \1_{\{R^n < \infty \}} \cdot J^{R,n+1} \big) \\
    & = \EV^{(n+1)}_x \big( f(X^{(n+1)}_{R^{(n)}}) \, \1_{\{R^{(n)} < \infty \}} \cdot J^{(n+1)} \big) \\
    & = \EV^{(n+1)}_x \big( K^n f \circ \pi^n \, \1_{\{R^{(n)} < \infty \}} \cdot J^{(n+1)} \big) \\
    & = \EV_x \big( K^n f \circ \pi^n \, \1_{\{R^n < \infty \}} \cdot J \big).
   \end{align*}
 This completes the proof, as $K^n f \circ \pi^n$ is $\sF_{R^n-}$-measurable by Lemma~\ref{lem:lift of transfer kernel, finite}.
\end{proof}

\section{Application to Pasting}

  \begin{figure}[tb] 
  \centering
  \begin{tikzpicture}[scale=0.45]
\draw[draw=none]  (-1*1.75,-4) -- (-1*1.75,2)  node[pos=.5, sloped] {$2\N \times E^{+1}$};
\draw[draw=none]  (-1.75*1.75,0) -- (-1.75*1.75,6)  node[pos=.5, sloped] {$(2\N - 1) \times E^{-1}$};
\draw[right] node at (6.0*1.75,-3.5) {$\pi$};
\draw[fill=gray, fill opacity=0.5] (0*1.75,3) ellipse (1.5 and 3);
\draw[fill=blue, fill opacity=0.5] (1*1.75,0) ellipse (1.5 and 3);
\draw[fill=gray, fill opacity=0.5] (2*1.75,3) ellipse (1.5 and 3);
\draw[fill=blue, fill opacity=0.5] (3*1.75,0) ellipse (1.5 and 3);
\draw[fill=gray, fill opacity=0.5] (4*1.75,3) ellipse (1.5 and 3);
\draw[fill=blue, fill opacity=0.5] (5*1.75,0) ellipse (1.5 and 3);
\node at (6.5*1.75,3) {$\cdots$};
\node at (6.5*1.75,0) {$\cdots$};
\node (A) at (3*1.75,-4) {$\bigcup_n \big( \{n\} \times E^{(-1)^n} \big)$};
\node (B) at (9*1.75,-4) {$E^{-1} \cup E^{+1}$};
\draw[->]  (A) -- (B);
\draw[fill=gray, fill opacity=0.5] (9*1.75,3) ellipse (1.5 and 3);
\draw[fill=blue, fill opacity=0.5] (9*1.75,0) ellipse (1.5 and 3);
\end{tikzpicture}
  \caption[Pasting of two processes via concatenation of alternating independent copies]
           {Construction of the pasting of two subprocesses $X^{-1}$, $X^{+1}$ on $E^{-1}$, $E^{+1}$, via concatenation of alternating subprocess copies on 
              $(2\N -1) \times E^{-1}$, $2\N \times E^{+1}$ respectively, and subsequent projection onto $E^{-1} \cup E^{+1}$.} \label{fig:alternating copies}
 \end{figure}
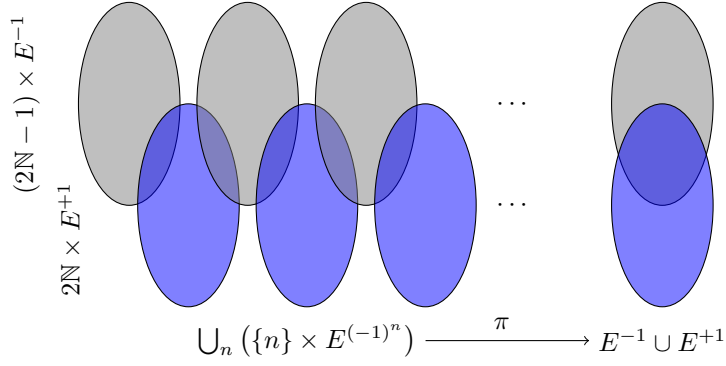

As described in section~\ref{sec:intro pasting}, we achieve the pasting of two right processes $X^{-1}$ and $X^{+1}$ on non-disjoint spaces $E^{-1}$ and $E^{+1}$
by introducing a counting coordinate, defining copies of the two processes on the disjoint spaces $\{n\} \times E^{(-1)^n}$, $n \in \N$, 
concatenating these processes to a process $X$ on $\N \times (E^{-1} \cup E^{+1})$, and then discarding the first coordinate by projecting to $\pi(X)$,
see~figure~\ref{fig:alternating copies}.
We now need to ensure that $\pi(X)$ is a right process.

\subsection{Mapping of the State Space}
In general, the state space transformation $\psi(X)$ of a (strong/right) Markov process~$X$ on a state space $E$ to a new state space $\hE$ via a surjective mapping 
$\psi \colon E \rightarrow \hE$ does not yield a (strong/right) Markov process. Heuristically speaking, the original process~$X$ needs to 
``behave identically''  on points of $E$ that are mapped together by $\psi$. A consistency condition with salvages the Markov property of $\psi(X)$ is 
found, e.g., in \cite[Theorem~10.13]{Dynkin65}, it reads
  \begin{align*}
    \forall B \in \hsE, x, x' \in E \text{ with } \psi(x) = \psi(x'):
      \quad \PV_x \big( X_t \in \psi^{-1}(B) \big) = \PV_{x'} \big( X_t \in \psi^{-1}(B) \big).
  \end{align*}
In the context of right processes the result is almost the same, flavored only by some measurability conditions.
It is found in \cite[Theorem (13.5)]{Sharpe88}:
\begin{theorem} \label{theo:mapping, standard}
  Let $X = \big( \Omega, \sG, (\sG_t, t \geq 0), (X_t, t \geq 0), (\Theta_t, t \geq 0), (\PV_x, x \in E) \big)$
  be a right process on a Radon space $E$ with semigroup $(T_t, t \geq 0)$ and resolvent $(U_\alpha, \alpha > 0)$.
  Let $(\hE, \hsE)$ be a Radon space and $\psi \colon E \rightarrow \hE$ be a mapping, satisfying the following conditions:
  \begin{enumerate}
    \item $\psi$ is $\sE/\hsE$-measurable and $\psi(E) = \hE$; \label{itm:mapping, standard i}
    \item $t \mapsto \psi(X_t)$ is a.s.\ right continuous in $\hE$; \label{itm:mapping, standard ii}
    \item for all $f \in b\cC_d(\hE)$ and all $t \geq 0$, there exists $g_t \in b\hsE$ such that  \label{itm:mapping, standard iii}
      \begin{align*} 
       T_t(f \circ \psi) = g_t \circ \psi.
      \end{align*}
  \end{enumerate}

  Define the transformed process $Y_t := \psi(X_t)$, $t \geq 0$, on 
  \begin{align*}
   \hO := \big\{\omega \in \Omega: t \mapsto \psi \big( X_t (\omega) \big) \text{ is right continuous in } \hE \big\},
  \end{align*}
  equipped with shift operators $\hT_t := \Theta_t$, $t \geq 0$, on $\hO$,
  and $\sigma$-algebras generated by~$Y$
  \begin{align*}
    \hsF^0 & := \sigma \big( \big\{ f(Y_t): f \in \hsE, t \geq 0 \big\} \big), \\
    \hsF^0_t & := \sigma \big( \big\{ f(Y_s): f \in \hsE, s \leq t \big\} \big), \quad t \geq 0,
  \end{align*} 
  and choose measures for $\hPV_y$, $y \in \hE$, by
  \begin{align} \label{eq:mapping, measure}
    \hPV_y & := \PV_x \text{ on $\hsF$}, \quad \text{for $x \in E$ with $\psi(x) = y \in \hE$}.
  \end{align}
  Furthermore, let $\hsF$, $(\hsF_t, t \geq 0)$ be the usual completion and augmentations of $\hsF^0$, $(\hsF^0_t, t \geq 0)$ respectively,
  relative to the family $(\hPV_y, y \in \hE)$.
  
  Then $Y = \big( \hO, \hsF, (\hsF_t)_{t \geq 0}, (Y_t)_{t \geq 0}, (\hT_t)_{t \geq 0}, (\hPV_y)_{y \in \hE} \big) =: \psi(X)$ is a right process on $\hE$.
\end{theorem}

As usual, property \ref{itm:mapping, standard iii} can be extended to all functions $f \in b\hsE$ by using the monotone class theorem and standard completion arguments 
(see \cite[Remarks~(13.6)]{Sharpe88}).
Because of this property, the definition of the measures~$\PV_y$ on~$\hsF$ in \eqref{eq:mapping, measure} 
is independent of the representatives chosen for $y = \psi(x)$, $x \in E$:  For any $f \in b\hsE$, $t \geq 0$, we have
 \begin{align*}
  \hEV_y \big( f(Y_t) \big)
  = \EV_x \big( f \big( \psi(X_t) \big) \big) 
  = T_t ( f \circ \psi ) (x)
  = g_t \circ \psi(x)
  = g_t (y).
 \end{align*}

Typically, the fundamental condition \ref{itm:mapping, standard iii} must be verified manually.
There is a Laplace-transformed version of this condition, which sometimes is easier to control, and which is more suitable in our context:

\begin{theorem} \label{theo:mapping, resolvent condition}
 In Theorem~\ref{theo:mapping, standard}, under \ref{itm:mapping, standard i} and \ref{itm:mapping, standard ii},
 condition \ref{itm:mapping, standard iii} is equivalent to
  \begin{enumerate}[label=(iii'), ref=(iii')] 
    \item for all $f \in b\cC_d(\hE)$ and all $\alpha > 0$, there exists $f_\alpha \in b\hsE$ such that  \label{itm:mapping, standard iiii}
        \begin{align*}
          U_\alpha(f \circ \psi) = f_\alpha \circ \psi.
        \end{align*}
  \end{enumerate}
\end{theorem}
\begin{proof}
 Assume that \ref{itm:mapping, standard i}, \ref{itm:mapping, standard ii} and \ref{itm:mapping, standard iii} hold.
 Then for $f \in b\cC_d(\hE)$, $\alpha > 0$, $x \in E_\Delta$,
  \begin{align*}
    U_\alpha(f \circ \psi) (x)
    & = \int_0^\infty e^{-\alpha t} \, T_t (f \circ \psi)(x) \, dt 
     = f_\alpha \circ \psi (x)
  \end{align*}
 holds with $f_\alpha := \int_0^\infty e^{-\alpha t} \, g_t \, dt \in b\hsE$ for $g_t \in b\hsE$ as given by \ref{itm:mapping, standard iii}.
 
 Now assume that \ref{itm:mapping, standard i}, \ref{itm:mapping, standard ii} and \ref{itm:mapping, standard iiii} hold.
 Let $f \in b\cC_d(\hE)$ and consider for every $\alpha > 0$ the function $f_\alpha \in b\hsE$ as given by \ref{itm:mapping, standard iiii} with $U_\alpha(f \circ \psi) = f_\alpha \circ \psi$.
 For $t = 0$, the function $g_0 = f$ satisfies $T_0(f \circ \psi) = g_0 \circ \psi$.
 For $t > 0$, we need to invert the Laplace transform, which is encoded in $(f_\alpha, \alpha > 0)$. 
 We first observe that $f_\alpha^{(k)} := \frac{\partial^k}{\partial \alpha^k} f_\alpha$ exists for all $k \in \N_0$,
 because for each $y \in \hE$, there is $x \in E$ with $\psi(x) = y$, 
 so
   \begin{align*}
     f_\alpha(y) = f_\alpha \big( \psi(x) \big) =  U_\alpha (f \circ \psi)(x)
   \end{align*}
 holds and $\alpha \mapsto U_\alpha (f \circ \psi)(x)$ is in $\cC^\infty(\R_{>0})$ (see \cite[Theorem XII.20]{DellacherieMeyerC}).
 Furthermore, for any $x \in E$, the function
   \begin{align*}
     t \mapsto T_t( f \circ \psi ) (x) = \EV_x \big( f \big( \psi(X_t) \big) \big)
   \end{align*}
 is a bounded and right continuous, as $f$ is bounded and continuous and $t \mapsto \psi(X_t)$ is right continuous by \ref{itm:mapping, standard ii}.
 Let $y \in \hE$, and choose any $x \in E$ with $\psi(x) = y$. 
 Then a general inversion formula\footnote{The inversion formula 
  \begin{align*}
   g(t) = \lim_{\varepsilon \downdownarrows 0} \lim_{\alpha \rightarrow \infty} \frac{1}{\varepsilon} \sum_{\alpha t < k \leq (\alpha + \varepsilon)t} \frac{(-1)^k}{k!} \, \alpha^k \, \varphi^{(k)}(\alpha), \quad t > 0,
  \end{align*}
 for the Laplace transform $\varphi(\alpha) = \int_0^\infty e^{-\alpha t} \, g(t) \, dt$ of a right continuous, bounded function $g \colon \R_+ \rightarrow \R$
 is given in \cite[Formula (4.14)]{Sharpe88} as part of an exercise with a reference to \cite[p.~232]{Feller71}. However, Feller only considers Laplace transforms
 of probability measures; in the general case the justification of the interchange of limits and integration, which is essential to Feller's proof, is more difficult and 
 can be found in \cite[Section~1.4]{Werner16}.} 
 for the Laplace transform of $t \mapsto T_t( f \circ \psi )$ yields
     \begin{align*}
     T_t( f \circ \psi ) (x)
     & = \lim_{\varepsilon \downdownarrows 0} \lim_{\alpha \rightarrow \infty} \frac{1}{\varepsilon} \sum_{\alpha t < k \leq (\alpha + \varepsilon)t} \frac{(-1)^k}{k!} \, \alpha^k \, U_\alpha^{(k)} (f \circ \psi)(x) \\
     & = \lim_{\varepsilon \downdownarrows 0} \lim_{\alpha \rightarrow \infty} \frac{1}{\varepsilon} \sum_{\alpha t < k \leq (\alpha + \varepsilon)t} \frac{(-1)^k}{k!} \, \alpha^k \, f_\alpha^{(k)}(y) \\
     & =: g_t(y) = g_t \circ \psi(x),
   \end{align*} 
 with the function $g_t \colon \hE \rightarrow \R$ as defined above being bounded as $\norm{g_t} = \norm{T_t ( f \circ \psi )}$ 
 and measurable due to the measurability of all $f^{(k)}_\alpha$, $\alpha > 0$, $k \in \N_0$.
\end{proof}

\subsection{Alternating Copies of Two Processes}
Let $X^{-1}$, $X^{+1}$ be two right processes with lifetimes $\zeta^{-1}$, $\zeta^{+1}$ on $E^{-1}$, $E^{+1}$ respectively, 
and $K^{-1}$, $K^{+1}$ be transfer kernels from $X^{-1}$ to $(X^{+1}, E^{+1})$ and from $X^{+1}$ to $(X^{-1}, E^{-1})$. 
Let $X$ be the concatenation, as described in section~\ref{sec:intro pasting}, of
 \begin{align*}
  X^n := \{n\} \times X^{(-1)^n}, \quad K^n := \delta_{n+1} \otimes K^{(-1)^n}, \quad n \in \N,
 \end{align*}
which by Theorem~\ref{theo:concatenation countable} is a right process on $\tE := \bigcup_n \{n\} \times E^{(-1)^n}$. 
Let $\pi \colon \tE \rightarrow E$, with $E := E^{-1} \cup E^{+1}$, be the projection onto the second coordinate.
We check the consistency conditions of Theorem~\ref{theo:mapping, resolvent condition}
to prove that the pasted process $\pi(X)$ is a right process on $E$.

\begin{proof}[Proof of Theorem~\ref{theo:alternating iterations}] 
 $\pi$ is clearly surjective. It is $\tsE / \sE$-measurable, as the preimage of $\pi$ reads
  \begin{align*}
   \pi^{-1}(B) = \big( (2 \N - 1) \times (B \cap E^{-1}) \big) \cup \big(2 \N \times (B \cap E^{+1}) \big), \quad B \in \sE.
  \end{align*}
 The right process $X$ is right continuous and the projection $\pi$ is continuous, so $\pi(X)$ is right continuous as well.
 By Theorem~\ref{theo:mapping, resolvent condition},
 it therefore suffices to prove that for all $\alpha > 0$, $f \in b\sE$, there exists $f_\alpha \in b\sE$ such that 
   $U_\alpha(f \circ \pi) = f_\alpha \circ \pi$ 
 holds true. As the process $X$ is constructed of alternating copies,
 we look at cycles of two revivals, that is, we examine for $(n,x) \in \tE$:
   \begin{align*}
    U_\alpha(f \circ \pi) \, (n,x)
     & = \sum_{m=0}^\infty \EV_{(n,x)} \Big( \1_{\{R^{n+2m-1} < \infty\}} \, \int_{R^{n+2m-1}}^{R^{n+2m+1}} e^{-\alpha t} \, f \circ \pi(X_t) \, dt \Big).
   \end{align*}  
   
 For $m = 0$, we decompose the partial resolvent at the revival time $R^{n}$ and obtain by employing the terminal time property of~$R^{n+1}$, 
 the strong Markov property of~$X$ at~$R^n$, and the revival formula of Theorem~\ref{theo:concatenation countable}:
   \begin{align*}
     & \EV_{(n,x)} \Big( \int_0^{R^{n+1}} e^{-\alpha t} \, f \circ \pi(X_t) \, dt \Big) \\
      & = \EV^{(-1)^n}_x \Big( \int_0^{\zeta^{(-1)^n}} e^{-\alpha t} \, f \big( X^{(-1)^n}_t \big) \, dt \Big) \\
      & \quad + \EV^{(-1)^n}_x \Big( \1_{\{\zeta^{(-1)^n} < \infty\}} \, e^{-\alpha \zeta^{(-1)^n}} \\
      & \quad \hspace*{5em} K^{(-1)^n} \EV^{(-1)^{n+1}}_{\, \cdot \,} \Big( \int_0^{\zeta^{(-1)^{n+1}}} e^{-\alpha t} \, f \big( X^{(-1)^{n+1}}_t \big) \, dt \Big) \Big) \\
      & =: g^{(-1)^n}_0(x).
   \end{align*}  
   
 For general $m \in \N_0$, we will show inductively that
   \begin{align} \label{eq:pasting map projection}
    \EV_{(n,x)} \Big( \1_{\{R^{n+2m-1} < \infty\}} \, \int_{R^{n+2m-1}}^{R^{n+2m+1}} e^{-\alpha t} \, f \circ \pi(X_t) \, dt \Big) = g^{(-1)^n}_m(x)
   \end{align}
 holds with $g^{-1}_m \in b\sE^{-1}$, $g^{+1}_m \in b\sE^{+1}$ being independent of $n \in \N$. The case $m = 0$ is already done.
 Assuming that \eqref{eq:pasting map projection} is proved for an $m \in \N_0$,
 we calculate for $m+1$, by using the same course of actions as above, as well as the definitions of the transfer kernels $K^n$:
    \begin{align*}
    & \EV_{(n,x)} \Big( \1_{\{R^{n+2(m+1)-1} < \infty\}} \, \int_{R^{n+2(m+1)-1}}^{R^{n+2(m+1)+1}} e^{-\alpha t} \, f \circ \pi(X_t) \, dt \Big) \\
    & = \EV_{(n,x)} \Big( \1_{\{R^{n} < \infty\}}  \, e^{-\alpha R^n} \, K^n \EV_{\, \cdot \,} \Big( \1_{\{R^{n+1} < \infty\}} \, e^{-\alpha R^{n+1}} \\
    & \hspace{4.6em} K^{n+1} \EV_{\, \cdot \,} \Big( \1_{\{R^{n+2m+1} < \infty\}} \, \int_{R^{n+2m+1}}^{R^{n+2m+3}} e^{-\alpha t} \, f \circ \pi(X_t) \, dt \Big) \circ \pi^{n+1} \Big) \circ \pi^n \Big)  \\
    & = \EV^{(-1)^n}_x \Big( \1_{\{\zeta^{(-1)^n} < \infty\}} \, e^{-\alpha \zeta^{(-1)^n}} \, K^{(-1)^n} \EV^{(-1)^{n+1}}_{\, \cdot \,} \Big( \1_{\{\zeta^{(-1)^{n+1}} < \infty\}} \, e^{-\alpha \zeta^{(-1)^{n+1}}} \\
    & \hspace{5.8em} K^{(-1)^{n+1}} \, \EV_{(n+2, \, \cdot \,)} \Big(\1_{\{R^{n+2m+1} < \infty\}} \, \int_{R^{n+2m+1}}^{R^{n+2m+3}} e^{-\alpha t} \, f \circ \pi(X_t) \, dt \Big) \Big) \Big) 
   \end{align*}
 Next, using the inductive assumption \eqref{eq:pasting map projection} and that $g^{(-1)^{n+2}}_m =  g^{(-1)^{n}}_m$, we get
   \begin{align*}
    & \EV_{(n,x)} \Big( \1_{\{R^{n+2(m+1)-1} < \infty\}} \, \int_{R^{n+2(m+1)-1}}^{R^{n+2(m+1)+1}} e^{-\alpha t} \, f \circ \pi(X_t) \, dt \Big) \\
    & = \EV^{(-1)^n}_x \Big( \1_{\{\zeta^{(-1)^n} < \infty\}} \, e^{-\alpha \zeta^{(-1)^n}} \, K^{(-1)^n} \EV^{(-1)^{n+1}}_{\, \cdot \,} \Big( \1_{\{\zeta^{(-1)^{n+1}} < \infty\}} \, e^{-\alpha \zeta^{(-1)^{n+1}}} \\
    & \hspace{26.2em}  K^{(-1)^{n+1}} \, g^{(-1)^{n}}_m \Big) \Big) \\
    & =: g^{(-1)^n}_{m+1}(x).
   \end{align*}
  
  Setting $g^{-1} := \sum_{m = 0}^\infty g^{-1}_m \in b\sE^{-1}$ and $g^{+1} := \sum_{m = 0}^\infty g^{+1}_m \in b\sE^{+1}$, we have proven that
  \begin{align*}
   U_\alpha(f \circ \pi) \, (n,x)
   = \begin{cases}
      g^{-1}(x), & \text{$n$ odd-numbered}, \\
      g^{+1}(x), & \text{$n$ even-numbered}
     \end{cases}
  \end{align*}
  holds for all $(n,x) \in \tE$, so the value of the resolvent $U_\alpha(f \circ \pi) \, (n,x)$ is independent of $n$ for all odd-numbered $n$, and for all even-numbered $n$.
  
  It remains to prove $g^{-1} = g^{+1}$ on $E^{-1} \cap E^{+1}$, which is equivalent to 
   \begin{align*} U_\alpha(f \circ \pi) \, (n_o,x) = U_\alpha(f \circ \pi) \, (n_e,x) \end{align*}
  for all $n_o \in (2 \N -1)$, $n_e \in 2 \N$, $x \in E^{-1} \cap E^{+1}$
  (because $(n_0,x) \notin E$ for $x \in E^{+1} \backslash E^{-1}$, and $(n_e,x) \notin E$ for $x \in E^{-1} \backslash E^{+1}$).
   
  Let $\tau_{-1}$ be the first entry time of $\pi(X)$ into $E^{-1} \backslash E^{+1}$, and $\tau_{+1}$ be the first entry time of $\pi(X)$ into $E^{+1} \backslash E^{-1}$.
  We synchronize the start of both processes by decomposing at the stopping time $\tau_{-1} \wedge \tau_{+1}$
  with Dynkin's formula \eqref{eq:Dynkins formula (resolvent)}:
  \begin{align*}
   U_\alpha(f \circ \pi) \, (n,x) 
   & = \EV_{(n,x)} \Big( \int_0^{\tau_{-1} \wedge \tau_{+1}} e^{-\alpha t} \, f \circ \pi(X_t) \, dt \Big) \\
   & ~ + \EV_{(n,x)} \big( e^{-\alpha (\tau_{-1} \wedge \tau_{+1})} \, U_\alpha (f \circ \pi) (X_{\tau_{-1} \wedge \tau_{+1}}) \big).
  \end{align*}
  $\tau_{-1} \wedge \tau_{+1}$ is the exit time of the process $X$ from $E^{-1} \cap E^{+1}$.
  The above formula will turn out to be independent of $n$ if
  the process' behavior on $E^{-1} \cap E^{+1}$ and its exit/entry behavior into $E \backslash (E^{-1} \cap E^{+1})$
  (represented by $e^{-\alpha (\tau_{-1} \wedge \tau_{+1})}$ and $X_{\tau_{-1} \wedge \tau_{+1}}$) are independent of $n$. 
  It has already been shown that this is the case for all odd-numbered $n$, and for all even-numbered $n$. 
  It remains to compare the odd-numbered and even-numbered starting processes, that is,
  the behavior of the original processes $X^{-1}$ and $X^{+1}$ together with the transfer kernels $K^{-1}$ and $K^{+1}$:
  
  For odd-numbered $n_o \in (2 \N -1)$, the starting process is $X^{(-1)^{n_o}} = X^{-1}$, living on~$E^{-1}$,
  so the process $\pi(X)$ starting at $(n_o, x)$ only enters $E^{+1} \backslash E^{-1}$ when the first subprocess dies. 
  Therefore,
  $\tau_{-1} \wedge \tau_{+1} = \tau_{-1} \wedge R^{n_o}$ holds true in this case, and using Dynkin's formula~\eqref{eq:Dynkins formula (resolvent)} again, we get
  \begin{align*}
   U_\alpha(f \circ \pi) \, (n_o, x)
   & = \EV_{(n_o, x)} \Big( \int_0^{\tau_{-1} \wedge R^{n_o}} e^{-\alpha t} \, f \circ \pi (X_t) \, dt \Big) \\
   & \quad   + \EV_{(n_o, x)} \big( e^{-\alpha \tau_{-1}} \, U_\alpha(f \circ \pi)(X_{\tau_{-1}}); \, \tau_{-1} < R^{n_o} \big) \\
   & \quad   + \EV_{(n_o, x)} \big( e^{-\alpha R^{n_0}} \, U_\alpha(f \circ \pi)(X_{R^{n_o}}); \, R^{n_o} \leq \tau_{-1} \big),
  \end{align*}
  where $R^{n_o} \leq \tau_{-1}$ can be replaced by $R^{n_o} < \tau_{-1}$, as equality only occurs if $R^{n_o} = \infty$.
  
  We have, $\PV_{(n_o, x)}$-a.s., $X_t = \big( n_o, X^{(-1)^{n_o}}_t \circ \pi^{n_o} \big)$ for all $t < R^{n_o} = \zeta^{(-1)^{n_o}} \circ \pi^{n_o} = \zeta^{-1} \circ \pi^{n_o}$,
  and $\tau^{-1}_{-1} \circ \pi^{n_o} < \zeta^{-1} \circ \pi^{n_o}$ if and only if $\tau_{-1} < R^{n_o}$, and in this case $\tau_{-1} = \tau^{-1}_{-1} \circ \pi^{n_o}$ holds true.
  Thus, the first part of the above decomposition reads
  \begin{align*}
   & \EV_{(n_o, x)} \Big( \int_0^{\tau_{-1} \wedge R^{n_o}} e^{-\alpha t} \, f \circ \pi (X_t) \, dt \Big) & \\
   & =       \EV_{(n_o, x)} \Big( \Big( \int_0^{\tau^{-1}_{-1}} e^{-\alpha t} \, f(X^{-1}_t) \, dt \Big) \circ \pi^{n_o} ; \, \tau_{-1} < R^{n_o} \Big) & \\
   & \quad + \EV_{(n_o, x)} \Big( \Big( \int_0^{\zeta^{-1}} e^{-\alpha t} \, f(X^{-1}_t) \, dt \Big)      \circ \pi^{n_o} ; \, R^{n_o} \leq \tau_{-1} \Big). & 
  \end{align*}
 As $f(X^{-1}_t) = f(\Delta) = 0$ for all $t > \zeta^{-1}$, we can replace the upper limit of the latter integration by $\tau^{-1}_{-1} \geq \zeta^{-1}$,
 in order to obtain
   \begin{align*}
   \EV_{(n_o, x)} \Big( \int_0^{\tau_{-1} \wedge R^{n_o}} e^{-\alpha t} \, f \circ \pi (X_t) \, dt \Big) 
    =       \EV_{(n_o, x)} \Big( \Big( \int_0^{\tau^{-1}_{-1}} e^{-\alpha t} \, f(X^{-1}_t) \, dt \Big) \circ \pi^{n_o} \Big). &
  \end{align*}
 Together with the process transfer at $R^{n_o}$ via $K^{(-1)^{n_o}} = K^{-1}$, and recalling that we already showed
 $U_\alpha(f \circ \pi) \, (n_o, \, \cdot \,) = g^{-1}$ and $U_\alpha(f \circ \pi) \, (n_o + 1, \, \cdot \,) = g^{+1}$, we get
\begin{equation} \label{eq:alternating copies, decomp1}
\begin{aligned}
   U_\alpha(f \circ \pi) \, (n_o, x)
   & = \EV^{-1}_x \Big( \int_0^{\tau^{-1}_{-1}} e^{-\alpha t} \, f(X^{-1}_t) \, dt \Big) \\
   & \quad   + \EV^{-1}_x \big( e^{-\alpha \tau^{-1}_{-1}} \, g^{-1}(X^{-1}_{\tau^{-1}_{-1}}); \, \tau^{-1}_{-1} < \zeta^{-1} \big) \\
   & \quad   + \EV^{-1}_x \big( e^{-\alpha \zeta^{-1}} \, K^{-1} g^{+1}; \, \zeta^{-1} < \tau^{-1}_{-1} \big). 
\end{aligned}
\end{equation}
  
 Analogously, we find that for any even-numbered $n_e \in 2 \N$,
\begin{equation} \label{eq:alternating copies, decomp2}
\begin{aligned} 
   U_\alpha(f \circ \pi) (n_e, x)
   & = \EV^{+1}_x \Big( \int_0^{\tau^{+1}_{+1}} e^{-\alpha t} \, f(X^{+1}_t) \, dt \Big) \\
   & \quad   + \EV^{+1}_x \big( e^{-\alpha \tau^{+1}_{+1}} \,g^{+1}(X^{+1}_{\tau^{+1}_{+1}}); \, \tau^{+1}_{+1} < \zeta^{+1} \big) \\
   & \quad   + \EV^{+1}_x \big( e^{-\alpha \zeta^{+1}} \, K^{+1} g^{-1}; \, \zeta^{+1} < \tau^{+1}_{+1} \big)
\end{aligned}
\end{equation}
 holds. Using the assumptions \ref{itm:alternating iterations  i} and \ref{itm:alternating iterations  ii} of the theorem, we conclude that
  \begin{align*}
   U_\alpha(f \circ \pi) \, (n_o, x) & = U_\alpha(f \circ \pi) \, (n_e, x),
  \end{align*}
 proving $U_\alpha(f \circ \pi) \, (n, x) = g^{\pm 1} \circ \pi(x)$ for all $x \in E$, $n \in \N$.
\end{proof}

\begin{proof}[Proof of Theorem~\ref{theo:identical iterations}]
 In case $X^{-1} = X^{+1}$ and $K^{-1} = K^{+1}$, each one of the summands of the decomposition~\eqref{eq:alternating copies, decomp1} 
 is equal to the corresponding summand of~\eqref{eq:alternating copies, decomp2}.
 Again, this yields $g^{-1} = g^{+1}$ and $U_\alpha(f \circ \pi) \, (n, x) = g^{\pm 1} \circ \pi(x)$ for all $x \in E$, $n \in \N$.
\end{proof}

\section*{Acknowledgements}
The main parts of this paper were developed during the author's Ph.D.\ thesis~\cite{Werner16} supervised by Prof.~J\"urgen~Potthoff, whose constant support the 
author gratefully acknowledges.

\bibliographystyle{amsplain}

\bibliography{concat_eps}

\end{document}